\documentclass[11pt, twoside]{amsart}
\usepackage{shadethm,framed}
\usepackage[utf8]{inputenc}
\usepackage[english]{babel}
\usepackage[papersize={21cm,29.7cm},left=3.3cm,right=3.3cm,top=3.17cm,bottom=3.160cm]{geometry}
\usepackage{enumerate}
\usepackage{listings}
\usepackage{amsthm, amsfonts, amssymb, amsmath}
\usepackage{graphicx}
\usepackage{xcolor}
\usepackage{color}
\usepackage{framed}
\usepackage{wallpaper}
\usepackage{tikz,tikz-cd}
\usepackage{pgf,tikz,pgfplots}
\usepackage{tcolorbox}
\usepackage{scalefnt}
\usepackage{mdframed}
\usepackage{booktabs}
\usepackage{caption}
\usepackage{makecell}
\usepackage{float}
\usepackage{thmtools}
\usepackage{thm-restate}

\usepgfplotslibrary{patchplots}
\usetikzlibrary{patterns, positioning, arrows,decorations.markings}

\definecolor{rose}{rgb}{0.96, 0.56, 0.76}
\usepackage[hidelinks,urlcolor=blue]{hyperref}
\hypersetup{
    colorlinks,
    linkcolor={blue!50!black},
    citecolor=rose,
    urlcolor={blue!80!black}
}
\usepackage{pgfplots}
\pgfplotsset{width=7cm, compat=1.10}
\usepgfplotslibrary{fillbetween}
\pgfmathdeclarefunction{poly}{0}{\pgfmathparse{sin(x)+x}}

\tikzcdset{arrow style=tikz,diagrams={>={Straight Barb[scale=0.8]},every arrow/.append style = -{Latex[scale=0.8]}}}

\tcbset{
    frame code={}
    center title,
    left=0pt,
    right=0pt,
    top=0pt,
    bottom=0pt,
    colback=red!70,
    colframe=white,
    width=\dimexpr\textwidth\relax,
    enlarge left by=0mm,
    boxsep=5pt,
    arc=0pt,outer arc=0pt,
    }

\newtheoremstyle{pourdef}
  {10pt}
  {10pt}
  {}
  {}
  {\bf}
  {.~}
  { }
  {}
\newtheoremstyle{pourth}{10pt}{10pt}{\em}{}{\sc}{.~}{ }{}
\newtheoremstyle{pourpp}{10pt}{10pt}{\em}{}{\bf \em}{.~}{ }{}
\newtheoremstyle{pourrk}{10pt}{10pt}{}{}{\em}{.~}{ }{}
\newtheoremstyle{pourlm}{10pt}{10pt}{\em}{}{\bf \em}{.~}{ }{}
\newtheoremstyle{pourco}{10pt}{10pt}{\em}{}{\bf \em}{.~}{ }{}

\definecolor{black}{cmyk}{1,1,1,1}
\definecolor{colordef}{rgb}{0.2,0.5,0.07}   
\definecolor{colorprop}{rgb}{0.2,0.1,0.5}

\definecolor{color1}{rgb}{0.6,0.4,0.8}
\definecolor{color1}{rgb}{0.9,0.6,0.4}
\definecolor{color1}{rgb}{0.36, 0.54, 0.66}

\definecolor{color2}{rgb}{0.2,0.1,0.5}
\definecolor{color2}{rgb}{0.91, 0.84, 0.42}
\definecolor{color2}{rgb}{0.87, 0.36, 0.51}
\definecolor{color2}{rgb}{0.4, 0.69, 0.2}
\definecolor{color2}{rgb}{0.84, 0.23, 0.24}

\definecolor{color3}{rgb}{1.0, 0.13, 0.32}
\definecolor{color3}{rgb}{0.54, 0.17, 0.89}

\definecolor{babypink}{rgb}{0.96, 0.76, 0.76}

\newtheorem{lem}{Lemma}[section]
\newtheorem{cor}[lem]{Corollary}
\newtheorem{thm}[lem]{Theorem}
\newtheorem{prop}[lem]{Proposition}

\usepackage{cleveref}

\declaretheorem[name=Theorem,numberwithin=section]{theorem}

\theoremstyle{definition}
\newtheorem{definition}[lem]{Definition}
\newtheorem{ex}[lem]{Example}

\newtheorem{fact}[lem]{Fact}
\newtheorem{rmk}[lem]{Remark}
\newtheorem{question}[lem]{Question}
\newtheorem*{question*}{Question}

\numberwithin{equation}{section}

\tikzcdset{arrow style=tikz,diagrams={>={Straight Barb[scale=0.8]},every arrow/.append style = -{Latex[scale=0.8]}}}

\newcommand{\N}{\mathbf N}
\newcommand{\Z}{\mathbf Z}

\newcommand{\R}{\mathbf R}

\newcommand{\Ein}{\mathbf{Ein}}
\newcommand{\Diams}{\mathbf{D}}

\newcommand{\dS}{\mathbf{dS}}
\newcommand{\D}{\mathbf D}

\newcommand{\C}{\mathbf C}
\renewcommand{\S}{\mathbf S}
\renewcommand{\H}{\mathbf H}
\renewcommand{\P}{\mathbf P}

\newcommand{\hol}{\mathsf{hol}}

\newcommand{\dev}{\mathsf{dev}}

\newcommand{\diag}{\operatorname{diag}}

\newcommand{\Conf}{\mathrm{Conf}}
\newcommand{\Iso}{\mathrm{Isom}}  

\newcommand{\Proj}{\operatorname{Proj}}

\newcommand{\Mark}{\mathsf{Mark}}

\newcommand{\Extr}{\operatorname{Extr}}

\renewcommand{\b}{\mathbf{b}}

\newcommand{\limhaus}{\operatorname{lim}_k}
\newcommand{\Grassd}{\operatorname{Gr}_p (\mathbf{R}^{2p})}
\def\O{\Omega}
\newcommand{\PO}{\operatorname{PO}}

\newcommand{\PGL}{\operatorname{PGL}}
\newcommand{\PSL}{\operatorname{PSL}}

\newcommand{\Gr}{\operatorname{Gr}_2(\mathbf{R}^4)}
\newcommand{\SLq}{\operatorname{PGL}(4, \mathbf{R})}
\newcommand{\PRSP}{\operatorname{\Lambda}^2 \mathbf{R}^4}
\newcommand{\PPRSP}{\mathbf{P}(\operatorname{\Lambda}^2 \mathbf{R}^4)}
\newcommand{\pp}{0}
\newcommand{\mm}{1}

\title[Proper almost-homogeneous domains in the Einstein Universe]{Proper almost-homogeneous domains of the Einstein universe}
\author{Adam Chalumeau and Blandine Galiay}

\begin{document}

\begin{abstract} The Einstein universe $\mathbf{Ein}^{p,q}$ of signature $(p,q)$ is a pseudo-Riemannian analogue of the conformal sphere; it is the conformal compactification of the pseudo-Riemannian Minkowski space. For $p,q \geq 1$, we show that, up to a conformal transformation, there is only one almost-homogeneous domain in $\mathbf{Ein}^{p,q}$ that is bounded in a suitable stereographic projection. This domain, which we call a diamond, is a model for the symmetric space of $\operatorname{PO}(p,1) \times \operatorname{PO}(1,q)$. We deduce a classification of closed conformally flat manifolds with proper development.

\end{abstract}

\maketitle

\section{Introduction}

In this paper we investigate domains (i.e.\ connected open subsets) in the Einstein universe, the pseudo-Riemannian analogue of the conformal sphere $\S^n$. We study a particular class of domains called \emph{proper}, that is domains that are bounded within a suitable stereographic projection. We answer the question of whether a proper domain $\O$ can be \emph{almost-homogeneous}, meaning that every point of the boundary $\partial \O$  can be approached by the orbit of a point $x \in \O$ under the action of the conformal group~$\Conf(\O)$.

\subsection{The Riemannian case. } \label{sect_riem_case} Any spherical ball of the conformal sphere $\S^n$ is a conformal model of the real hyperbolic space $\H^n$. This provides an example of a proper symmetric domain of the conformal sphere. When $n=2$, slight deformations of a cocompact lattice $\Gamma <\PSL(2,\R)$ into $\PSL(2,\C)$ remain discrete and faithful, acting cocompactly on a proper domain bounded by a topological circle. These examples are called \emph{quasi-Fuchsian}. Similar techniques create examples of lattices $\Gamma$ of $\PO(1,n+1)$ acting properly discontinuously and cocompactly on a proper domain of $\S^n$, for $n\geq 3$.\linebreak In fact, \cite{Apanasov} presents examples of deformations of lattices $\Gamma$ of $\PO(1,n)=\Iso(\H^n)$ into $\PO(1,n+1)=\Conf(\S^n)$, corresponding to deforming the conformally flat structure of a finite volume complete hyperbolic $n$-manifold. Given an almost-homogeneous domain $\Omega\subset\S^n$, either it is a spherical ball or its automorphism group is Zariski dense in $\PO(1,n+1)$. In particular, the examples of \cite{Apanasov} feature a discrete conformal group.

\subsection{Higher signature case} Let $p,q\geq 1$. The \emph{Einstein universe} $\Ein^{p,q}$ is a compact pseudo-Riemannian manifold of signature $(p,q)$, where $p$ and $q$ are the number of negative and positive signs respectively. It can be regarded as the boundary of the $(p,q+1)$ pseudo-Riemannian hyperbolic space $\H^{p,q+1}$. A notion of stereographic projection  exists from a dense open subset of $\Ein^{p,q}$ onto the \emph{Minkowski space}~$\R^{p,q}$, the pseudo-Riemannian affine space of signature $(p,q)$. In Lorentzian signature ( i.e. for $p=1$), one can construct a proper symmetric domain using the causal structure. Given two chronologically related points $a,b\in\R^{1,n-1}$, the \emph{diamond} defined by $a$ and $b$ is the intersection of the chronological future of $b$ with the chronological past of $a$. This defines a bounded domain conformally equivalent to $\R\times\H^{n-1}$ endowed with the product metric $-dt^2\oplus g_{\H^{n-1}}$. The conformal group of a diamond is isomorphic \linebreak to $\Iso(\R)\times\Iso(\H^{n-1})$. For general signature $(p,q)$, the space $\H^p\times \H^q$ endowed with the metric $-g_{\H^{p}}\oplus g_{\H^{q}}$ embeds as a proper homogeneous domain of the Einstein universe, see Section~\ref{sect_higher_diamonds}. One way to construct this embedding is by examining the open orbits of the action of $\PO(p,1)\times\PO(1,q)\subset\PO(p+1,q+1)$ on $\Ein^{p,q}$. All these domains are conjugate in $\PO(p+1,q+1)$ and referred to as \emph{diamonds}. A purely causal definition exists for them, see \cite{Romeo}. In contrast with the Riemannian case, for $p,q \geq 1$, we prove the following.

\begin{thm}
\label{mainth}
    Any proper almost-homogeneous domain of $\Ein^{p,q}$ is a diamond.
\end{thm} 

Weakening the assumptions on the domain provides more flexibility. There are indeed many examples of nonproper domains of $\Ein^{p,q}$ that admit a cocompact and properly discontinuous action of a discrete subgroup of $\PO(p+1, q+1)$; for example the \emph{affine charts}, i.e.\ subsets of the form $\Ein^{p,q} \smallsetminus C(x)$ with $x \in \Ein^{p,q}$ (see Section~\ref{geom_Ein_universe}). Other examples are given by cocompact domains of discontinuity of Gromov hyperbolic groups acting via $P_{r+1}$-Anosov representations with values in $\PO(p+1, q+1)$, where $r=\min(p,q)$ and $P_{r+1}$ is the stabilizer in  $\PO(p+1, q+1)$ of a maximal totally isotropic subset of $\R^{p+1,q+1}$, see e.g.\ \cite{Frances_LKG, guichard2012anosov}. These examples are never proper.

\subsection{An application to conformally flat structures. } When $p$ and $q$ are different, cocompact lattices of $\PO(p,1)\times\PO(1,q)$ are all reducible and virtually products of cocompact lattices. This fact and Theorem~\ref{mainth} enable us to fully classify closed conformally flat manifolds with \emph{proper development}, that is closed manifolds $M$ such that the image of the developing map $\dev:\widetilde{M}\to\Ein^{p,q}$ is proper.

\begin{thm}
    \label{maincor}
    Let $1 \leq p < q$ with $(p,q)\neq (2,3)$ and $M$ be a closed connected conformally flat pseudo-Riemannian manifold with proper development. Then, up to a finite cover, the manifold~$M$ is conformally equivalent to $$(-\Sigma^p)\times\Sigma^q,$$
    where $\Sigma^p$ and $\Sigma^q$ are closed Riemannian hyperbolic manifolds. In Lorentzian signature, i.e. for $p=1$, the manifold $M$ is (up to finite cover) conformally equivalent to the product $(-\S^1)\times \Sigma$, where $\Sigma$ is a closed Riemannian hyperbolic manifold.
\end{thm}

In particular, every closed manifold with proper development is \emph{Kleinian}, i.e. the developing map is a diffeomorphism onto its image. In general, the developing map of a closed conformally flat manifold may not even be a covering map onto its image. The Ehresmann--Thurston deformation principle provides Lorentzian conformally flat structures on $M=\S^1\times\Sigma_g$, where $\Sigma_g$ is a closed Riemannian surface of genus $g\geq 2$, such that the holonomy of $M$ is non-discrete in $\PO(2,3)$ (see \cite[Sect. 10.3.4]{Frances_LKG}). \\

\subsection{Divisible convex sets in flag manifolds.} The theory of divisible convex subsets of the projective space $\mathbf{P}(\mathbf{R}^{n+1})$ generalizes the theory of closed real hyperbolic manifolds and has been widely studied since Benzecri's thesis \cite{benzecri1960varietes}. The objects of study are domains that are bounded in some affine chart and that admit a cocompact action of a discrete subgroup of $\PGL(n+1, \mathbf{R})$.

A number of authors have investigated their general properties (see e.g.\ \cite{vinberg1965structure, koszul1968deformations, vey1970automorphismes, benoist2001convexes, benoist2003convexes}) and the construction of examples (see e.g.\ \cite{vinberg1967quasi, johnson1987deformation, ballas2018convex, blayac2023divisible}). The case where the set is strictly convex is well understood, while the case of non-strictly convex sets remains a subject of questions and recent research \cite{islam2019rank, choi2020convex, zimmer2020higher, blayac2021boundary}; see \cite{benoist2008survey} for a survey on this theory and references. 

There are many examples of divisible convex sets that are nonsymmetric, i.e.\ that are not Riemannian symmetric spaces embedded in the projective space, so that the theory does not reduce to the one of Riemannian symmetric spaces.

The projective space is an example of a \emph{flag manifold}, i.e.\ a compact manifold $G/P$, where $G$ is a noncompact semisimple Lie group and $P$ a parabolic subgroup of $G$. Some notions of projective geometry can be adapted to study divisible convex sets in general flag manifolds (see \cite{Zimpropqh}); in particular, there is a notion of \emph{properness} in flag manifolds. The question of whether the theory of divisible convex sets generalizes to other flag manifolds was asked by W. van Limbeek and A. Zimmer.  One issue is to determine whether there are examples of nonsymmetric divisible domains.

\begin{question} \label{question_Lim_Zim}\cite{van2019rigidity}
    Given some flag manifold $G/P$, are all divisible convex set in $G/P$ symmetric?
\end{question}

Although we saw in Section~\ref{sect_riem_case} that there are examples of rank-one Lie groups where Question \ref{question_Lim_Zim} has a negative answer, in higher rank it received a positive partial response in several cases: Limbeek--Zimmer proved rigidity of proper divisible domains in the flag manifolds $\Grassd$ that are convex in some affine chart \cite{van2019rigidity}, and Zimmer proved in \cite{zimmer2013rigidity} that every divisible domain of $\mathbf{P}(\mathbf{C}^n)$ with $C^1$ boundary is projectively isomorphic to the complex unit ball. The case of an arbitrary flag manifold~$G/P$ where $P$ is a nonmaximal proper parabolic subgroup of $G$ also received a positive answer \cite{Zimpropqh}. 

Theorem~\ref{mainth} gives a positive answer to Question \ref{question_Lim_Zim} for the flag manifolds $\Ein^{p,q} = \PO(p+1, q+1)/ P_1$, where $P_1 \leq \PO(p+1, q+1)$ is the stabilizer of an isotropic line of $\mathbf{R}^{p+1, q+1}$. Some exceptional Lie groups isomorphisms occurring in low dimension give:

\begin{cor}\label{cor_rigidity}
$(1)$ Let $\Omega\subset\Gr$ be a proper domain and assume that there exists a subgroup of $\PGL(4, \R)$ preserving and acting almost-homogeneously on $\O$. Then the group of all the elements of $G = \PGL(4, \R)$ preserving $\O$ is conjugate to $\PO(2,2)$.

$(2)$ Let $\mathcal{F}$ be one of the two connected components of the space of maximal totally isotropic subspaces of $\R^{4,4}$. Let $\Omega\subset\mathcal{F}$ be a proper domain and assume that there exists a subgroup of $G = \PO(4,4)$ preserving $\O$ and acting almost-homogeneously on it. Then the group of all the elements of $\PO(4, 4)$ preserving $\O$ is conjugate to $\PO(3,1) \times \PO(1,3)$.
\end{cor}

The first part of Corollary \ref{cor_rigidity} improves the main result of \cite{van2019rigidity} for $p=2$, both by replacing the divisibility assumption by a weaker almost-homogeneity assumption, and by removing the assumption that $\O$ is \emph{convex in an affine chart}. J. Danciger and W. van Limbeek informed us that they have obtained a similar result to Corollary~\ref{cor_rigidity}.(1) using different methods.

Theorem~\ref{mainth} and Corollary~\ref{cor_rigidity} imply in particular that every almost-homogeneous domain of the considered flag manifolds $G/P$ is divisible. This fact does not hold for every flag manifold. For instance, there exist properly convex domains of the projective space that are almost-homogeneous and not divisible, and not even quasi-homogeneous
(i.e.\ with cocompact automorphism group), see \cite{ballas2018convex}. 

\begin{rmk}
    The starting point of this project was the case of $\Ein^{1,n-1} = \PO(2,n)/P_1$ (see Section~\ref{sect_proof_loren}). It was then generalized simultaneously to $\Ein^{p,q} = \PO(p+1, q+1)/ P_1$ in the present paper, and to Shilov boundaries of Hermitian symmetric spaces of tube type in \cite{rigidityShilov}, providing in both cases a positive answer to Question~\ref{question_Lim_Zim}. 
\end{rmk}

\subsection{Outline of the proof} 

\subsubsection{Invariant distance function.} A key ingredient for the study of proper domains $\O$ in $\Ein^{p,q}$ is a natural conformally invariant distance function $\delta^\O$, a pseudo-Riemannian analogue of the Hilbert metric. Such a distance function was defined by Markowitz in \cite{markowitz_1981} for pseudo-Riemannian manifolds, as a generalization of Kobayashi's projectively invariant pseudodistance \cite{kob}. Its definition relies on the notion of projective parameters of lightlike geodesics. We define this distance function for domains of the Einstein universe and study its basic properties in Section~\ref{kobayashimetric}. In Section~\ref{properness}, we show that $\delta^\O$ is proper and geodesic whenever $\Omega$ is proper and almost-homogeneous.

\subsubsection{Dual convexity.} A first step in the proof of Theorem \ref{mainth} is to investigate the convex structure of an almost-homogeneous domain $\O$. In real projective space, any proper almost-homogeneous domain is convex \cite[Prop.\ 3.25]{kob}. In \cite{Zimpropqh}, Zimmer defines \emph{dual convexity} as a generalization of convexity in the real projective setting, using the characterization of convexity of bounded domains by the existence of supporting hyperplanes at any point of the boundary. He proves that every proper quasi-homogeneous domain of a flag manifold $G/P$ is dually convex (see \cite[Cor.\ 9.3]{Zimpropqh}), and we will use a slightly stronger version of that statement for almost-homogeneous domains. 

For the case of Lorentzian geometry,  we relate the notion of dual convexity with the standard notion of causal convexity, see Section~\ref{convexity}. This enables us to give a short proof of Theorem \ref{mainth} in the Lorentzian signature in Section~\ref{endoftheproof}. In fact,  we show that proper dually convex domains are precisely the proper causally convex maximal domains of the Minkowski space in the sense of \cite{smaï2023enveloping}.

\subsubsection{Photon-extremal points}\label{sect_outline_extr}  To prove Theorem~\ref{mainth}, we define a notion of \emph{photon-extremality} for points of $\partial \O$, similar to that of extremality for domains of real projective space. Following the strategy of Limbeek--Zimmer \cite{van2019rigidity}, we investigate the dynamics of $\Conf(\O)$ near extremal points. An analysis of the dynamics of divergent sequences of $\PO(p+1,q+1)$, using the Markowitz distance function as well as dual convexity, shows that extremal points satisfy a strong geometric property, see Proposition~\ref{extremal_implique_conedisjoint}. 
This will split extremal points into two $\Conf(\O)$-invariant subsets: spacelike-extremal and timelike-extremal points (Section$~$\ref{spacelike_timelike_etremal}). Using the projective model of $\Ein^{p,q}$, this will impose that the conformal group $\Conf(\O)$ preserves a splitting $\R^{p+1,q+1}=V_\pp\oplus V_\mm$. The study of the signature of $V_\pp$ and $V_\mm$ will show that $\Conf(\O)=\PO(p,1)\times\PO(1,q)$ and that $\Omega$ is a diamond, see Section \ref{sect_proof_any_sign}.

\subsection{Organization of the paper} In Section \ref{sect_prelimi}, we recall well-known facts of conformal geometry and we introduce the Einstein universe. In Section \ref{sect_proper_domains}, we define the notions of properness, almost-homogeneity and dual convexity, and introduce the only proper symmetric domains of $\Ein^{p,q}$, namely, diamonds. In Section \ref{kobayashimetric}, we investigate the properties of the Markowitz pseudodistance on a domain $\O$ of $\Ein^{p,q}$, which turns out to be a distance function as soon as $\O$ is proper. In Section \ref{sect_dyn_boun}, we define \emph{photon-extremal points} of the boundary of a proper domain, and use the Markowitz distance function to prove that they satisfy a strong geometric property in the almost-homogeneous case. This geometric property allows us to finish the proof of Theorem \ref{mainth} in Lorentzian signature in Section \ref{sect_proof_loren} and in any signature in Section \ref{sect_proof_any_sign}. In Section \ref{sect_proof_mainth2} we prove Theorem \ref{maincor}. Finally, in Section \ref{sect_except_iso} we prove Corollary \ref{cor_rigidity}.

\subsection*{Acknowledgements} The authors would like to thank Charles Frances and Fanny Kassel for their support and valuable guidance in the research leading to the results presented in this paper. They would also like to thank Roméo Troubat for sharing his construction of diamonds and his inspiring insights in higher signature, and Yves Benoist for initially suggesting that they explore rigidity in the Einstein Universe of signature $(1, 2)$, which was the starting point for this work. The first-named author is also grateful to Pierre Py for the valuable discussions they shared on discrete subgroups of Lie groups. The second-named author also expresses gratitude to Jeffrey Danciger and Wouter van Limbeek for valuable insights on Question \ref{question_Lim_Zim}. Finally, the authors would like to thank the organizers of the conference \emph{Diverse Aspects of Groups, Geometry and Dynamics}, held in Heidelberg in September 2023, where this collaboration started.

\section{Preliminaries}\label{sect_prelimi} 
In this section, we recall some notions of conformal geometry and fix some notations.

\subsection{General notions of conformal geometry. } \label{sect_conf_geom}
A \emph{conformal manifold} is a manifold $M$ equipped with a \emph{conformal class} $[g]$, where $g$ is a pseudo-Riemannian metric on $M$, and
$$ [g] = \left\{ e^f \cdot g \, \vert \, f \in C^\infty(M) \right\}. $$
The signature of $g$ will be denoted by $(p, q)$, where $p$ refers to the number of ``$-$'' and $q$ to the number of ``$+$''. If $g$ has signature $(p,q)$, we say that $(M,[g])$ is a conformal manifold of signature $(p,q)$. At every point $x \in M$, the bilinear form $g_x$ can be written as
\[ g_x(v, v) = -(v_1)^2 \dots -(v_p)^2 + (v_{p+1})^2 + \dots + (v_{p+q})^2, \]
in a suitable basis of $T_xM$. A tangent vector $v \in TM$ is \emph{timelike} (resp.\ \emph{lightlike, spacelike}) if $g(v, v)$ is negative (resp.\ null, resp.\ positive); it is \emph{causal} if it is either timelike or lightlike. This enables us to talk about \emph{timelike curves} in $M$ (resp.\ \emph{causal}, resp.\ \emph{lightlike}, resp.\ \emph{spacelike}).

A smooth map $\varphi: (M, [g_M]) \to (N, [g_N])$ is \emph{conformal} if $\varphi^* g_N \in [g_M]$, which is the same as saying that $\varphi$ sends causal curves to causal curves and spacelike curves to spacelike curves. We will write $\Conf(M)$ for the group of conformal automorphisms of $M$, and we will call it the \emph{conformal group} of $M$.

\subsubsection{Invariance of lightlike geodesics. }In general, different metrics within the same conformal class of a conformal manifold will define geometrically distinct geodesics. However, the notion of unparametrized lightlike geodesic has a conformal meaning.

\begin{theorem}
    Two conformally equivalent metrics have the same lightlike pregeodesics.
\end{theorem}

In fact, lightlike geodesics carry a invariant family of parameters called \emph{projective}, see \cite{markowitz_1981}. We won't need that degree of generality however.


\subsubsection{Conformal spacetimes. }\label{spacetime}When the metric $g$ has signature $(1,n-1)$ where $n=\dim(M)$, we say that $g$ is a Lorentzian metric and that $(M, [g])$ is a \emph{conformal Lorentzian manifold}. In every tangent space of a conformal Lorentzian manifold, the set of nonzero causal vectors has two connected components. If there exists a global continuous choice of such connected components on $M$, we say that $M$ is \emph{time-orientable}. A continuous choice $x\mapsto F_x\subset T_xM$ of connected components of causal vectors is called a \emph{time-orientation}. Causal vectors belonging (resp.\ not belonging) to these components are called \emph{future directed} (resp.\ \emph{past directed}).  
A conformal Lorentzian manifold that is time-oriented is called a \emph{conformal spacetime}.
Given a conformal spacetime $(M, [g])$ and $x \in M$, we write $I^+(x)$ (resp.\ $J^+(x)$) for the union of all endpoints of future timelike (resp.\ causal) curves starting at $x$. Similarly, we write $I^-(x)$ (resp.\ $J^-(x)$) for the union of all endpoints of past timelike (resp.\ causal) curves starting at $x$.

\subsubsection{The Minkowski space and the Einstein universe. } For arbitrary $p,q \in \mathbf{N}$, we will write $\R^{p,q}$ for the vector space $\R^{p+q}$ endowed with a bilinear form $\b$ of signature $(p,q)$. This is a complete flat pseudo-Riemannian manifold called the \emph{Minkowski space}. We will also write $\R^{p,q}$ for the associated conformal manifold $\R^{p,q} = (\R^{p+q},[\b])$ and we will also call it the \emph{Minkowski space}. 
 
The \emph{Einstein universe} is the space, denoted by $\Ein^{p,q}$, defined as 
\begin{equation*}
    \Ein^{p,q} = \{\mathbf{P}(v) \in \mathbf{P}(\mathbf{R}^{p+1,q+1}) \mid \b(v,v)  = 0\},
\end{equation*}
where $\mathbf{P}(v)$ is the line of $\mathbf{R}^{p+1,q+1}$ generated by $v \neq 0$. In other words $\Ein^{p,q}$ is the set of isotropic lines in $\R^{p+1,q+1}$. It is endowed with a natural pseudo-Riemannian conformal class $\left[\b\vert_{\Ein^{p,q}}\right]$ coming from the metric on $\R^{p+1,q+1}$. The Einstein universe is a compact pseudo-Riemannian conformal manifold of signature $(p,q)$. It admits a 2-sheeted conformal cover by $(\S^p \times \S^q,[-g_{\S^{p}} \oplus g_{\S^{q}}])$ whose nontrivial deck transformation is $(x, y) \mapsto (-x, -y)$. The natural action of $\PO(p+1,q+1)$ on $\Ein^{p,q}$ is by conformal automorphisms, and the conformal group of $\Ein^{p,q}$ coincides with $\PO(p+1,q+1)$. More generally, we have the following fundamental result which is due to Liouville in the Riemannian setting (see \cite{france_liouville}):
\begin{theorem}
    Let $U,V$ be connected domains of $\Ein^{p,q}$, and let $\varphi:U\to V$ be a conformal map. If $p+q\geq 3$, then there is a unique element $g\in\PO(p+1,q+1)$ such that $g\vert_U=\varphi$.
\end{theorem}
In particular, for a connected domain $\O\subset\Ein^{p,q}$, the conformal group of $\O$ is precisely the subgroup of $\PO(p+1,q+1)$ of all transformations $g$ preserving $\O$.

\subsubsection{Conformally flat manifolds. }In this article, we will restrict our study to \emph{conformally flat manifolds}.
We say that a conformal manifold $(M,[g])$ is \emph{conformally flat} if it is locally conformally equivalent to an open subset of the Minkowski space. A consequence of Liouville's theorem is that  conformally flat pseudo-Riemannian manifolds inherit a canonical $(\PO(p+1,q+1),\Ein^{p,q})$-structure (for basic properties of $(G,X)$-structures, see \cite{Thurston}). This implies the following.

\begin{prop}
\label{Developingmap}
    Let $M$ be a conformally flat manifold. Then there exists a conformal map $\dev:\widetilde{M}\to \Ein^{p,q}$. If $\dev^\prime:\widetilde{M}\to \Ein^{p,q}$ is another such conformal map, then there is a unique element $g\in\PO(p+1,q+1)$ such that $\dev^\prime=g\circ\dev$.  
\end{prop}

A map $\dev:\widetilde{M}\to \Ein^{p,q}$ as in Proposition~\ref{Developingmap} is called a \emph{developing map} for $M$. A $(\PO(p+1,q+1),\Ein^{p,q})$-structure on $M$ is also encoded by the data of a \emph{holonomy morphism} $\hol:\pi_1(M)\to\PO(p+1,q+1)$. However, we will not use $\hol$ in this article. 

\subsection{Geometry of the Einstein universe}\label{geom_Ein_universe}

We briefly recall the basic tools for studying $\Ein^{p,q}$. For a more general overview, see \cite{Frances_LKG} and \cite{Primer}. 

\subsubsection{Photons and lightcones. }We write $C^{p+1,q+1}$ for the isotropic cone of $\R^{p+1,q+1}$ and $\pi: C^{p+1,q+1} \to \Ein^{p,q}$ for the projectivization map. The lightlike geodesics of $\Ein^{p,q}$ are precisely the curves $\pi(\Pi)$, where $\Pi \subset \R^{p+1,q+1}$ is a totally isotropic 2-plane. Such a curve is called a \emph{photon}. The union of all photons passing through a point $x \in \Ein^{p,q}$, called the \emph{lightcone} of $x$, will be denoted by $C(x)$. The lightcone $C(x)$ is the image under $\pi$ of $C^{p+1,q+1} \cap x^\perp$. It is a singular hypersurface of signature $(0, -,\dots, -, +\dots, +)$ that is foliated by photons. In fact, the set $C(x) \smallsetminus \{x\}$ is a smooth submanifold that is conformal to $(\R \times \Ein^{p-1,q-1}, [0 \oplus(- g_{\S^{p-1}} )\oplus g_{\S^{q-1}}])$, where $(- g_{\S^{p-1}} )\oplus g_{\S^{q-1}}$ is the quotient metric on $\Ein^{p-1,q-1}=\S^{p-1}\times\S^{q-1}/_{(x,y)\sim(-x,-y)}$.

\begin{rmk}
    Given a point $x \in \Ein^{p,q}$, the lightcone of $x$ coincides with the nontransverse the set $Z_x$ in the sense of \cite[Sect.\ 1.6]{Zimpropqh}. In the setting of general flag manifolds, the set $Z_x$ is classically called a \emph{maximal proper Schubert subvariety} of $\Ein^{p,q}$.
\end{rmk}

\subsubsection{Affine charts. }Given a point $x\in \Ein^{p,q}$, the open set $\Ein^{p,q} \smallsetminus C(x)$ can be conformally identified with $\mathbf{R}^{p,q}$. An identification $\Ein^{p,q} \smallsetminus C(x) \simeq \mathbf{R}^{p,q}$ is called an \emph{affine chart} or a \emph{stereographic projection}. With such an identification, the open set $\Ein^{p,q} \smallsetminus C(x)$ inherits a canonical affine structure, independent of the conformal identification $\Ein^{p,q} \smallsetminus C(x) \simeq \mathbf{R}^{p,q}$. This identification is the reason why $\Ein^{p,q}$ is sometimes referred to as the \emph{conformal compactification of }$\mathbf{R}^{p,q}$: the Einstein universe is a compact conformal manifold containing a conformal copy of the Minkowski space as an open and dense subset.

\subsubsection{Cartan decomposition of $\PO(p+1,q+1)$. } Let $1\leq p\leq q$ and let $G=\PO(p+1,q+1)$. We fix a basis in which the metric of $\R^{p+1,q+1}$ takes the expression 
$\b(v,v)=v_0v_{p+q+1}+\dots+v_pv_{q+1}+(v_{p+1})^2+\dots+(v_{q})^2,$
and we write elements of $G$ in that basis.
Let us recall the~$KA^+K$ decomposition of~$G$. Let $A^+$ be the set of diagonal matrices 
$$\diag\left(\lambda_0,\dots,\lambda_p,1,\dots,1,\lambda_p^{-1},\dots,\lambda_0^{-1}\right)$$
such that $\lambda_0\geq\lambda_1\geq\dots\geq\lambda_p\geq 1$. There exists a maximal compact subgroup~$K$ of~$G$ such that any element~$g \in G$ can be written $g=\kappa a\kappa^\prime$ with $\kappa,\kappa^\prime\in K$ and $a\in A^+$. With this convention, the element $a$ is uniquely determined by $g$.

\subsubsection{Dynamics on the Einstein universe. }\label{dynamics} Let $(g_k)$ be a sequence in $\PO(p+1,q+1)$ and let $\lambda_0(k)\geq\dots\geq\lambda_p(k)\geq 1$ be the associated sequences of eigenvalues for the matrix $a_k\in A$ in the Cartan decomposition of $g_k$ as above. It is possible to describe geometrically the dynamic behaviour of $(g_k)$ on $\Ein^{p,q}$ in terms of the sequences $(\lambda_i(k))$ (see \cite{Frances_LKG} in the Lorentzian case and \cite{chalumeau} for general signature). Since the general setting is slightly technical, we decided not to introduce it in this paper. We will only use the notion of \emph{contracting} sequence in the proof of Proposition \ref{extremal_implique_conedisjoint}. We say that~$(g_k)$ is \emph{contracting} if there exist two points $x,y\in\Ein^{p,q}$ (possibly equal) such that the restriction of $(g_k)$ to $\Ein^{p,q}\smallsetminus C(y)$ converges uniformly on compact subsets to the constant map equal to $x$. In that case, we will also say that $(g_k)$ is $(x,y)$\emph{-contracting}. In terms of the Cartan decomposition, this contracting condition is equivalent to saying that $\lambda_0(k)/\lambda_i(k)\to+\infty$ as $k\to+\infty$, for every $i\in\{1,\dots,p\}$. 

\begin{fact}[See e.g.\ {\cite[Appendix A]{weisman2022extended}}]
    \label{fact_dyna}
    Let $(g_k)$ be a sequence of $\PO(p+1,q+1)$. Assume that there exists a compact subset $\mathcal{C}\subset\Ein^{p,q}$ with nonempty interior such that $(g_k(\mathcal{C}))$ converges to a point with respect to the Hausdorff topology. Then $(g_k)$ is contracting. 
\end{fact}

\section{Proper domains of the Einstein universe}\label{sect_proper_domains}

\subsection{First definitions}\label{sect_hom_dom} 
Let $ \O \subset \Ein^{p,q}$ be a domain, i.e.\ a connected open subset. We say that $\O$ is \emph{proper} if there exists some $x \in \Ein^{p,q}$ with $\overline{\O} \cap C(x) = \emptyset$. In other words, the domain $\O$ appears as a bounded domain in $\mathbf{R}^{p,q}$ by means of the affine chart $\Ein^{p,q} \smallsetminus C(x)$. This notion coincides with the notion of properness of $\O$ as a subset of the flag manifold $\Ein^{p,q}$ recalled in the introduction.

\begin{definition}
    We say that a domain $\O\subset\Ein^{p,q}$ is \emph{quasi-homogeneous} (resp.\ \emph{homogeneous}, resp.\ \emph{divisible}) if there exists some compact subset $\mathcal{K} \subset  \O$ such that $\O = \Conf(\O) \cdot \mathcal{K}$ (resp.\ if $\Conf(\O)$ acts transitively on $\O$, resp.\ if there exists a discrete subgroup $\Gamma\subset\Conf(\O)$ and a compact subset $\mathcal{K} \subset  \O$ such that $\O = \Gamma \cdot \mathcal{K}$). We say that $\O$ is \emph{almost-homogeneous} if for every point $a \in \partial \O$ there exist $(g_k) \in \Conf(\O)^{\mathbf{N}}$ and a sequence $(x_k) \in \O^{\mathbf{N}}$ converging in $\O$ such that $g_k \cdot x_k \rightarrow a$. 
\end{definition}

Divisibility and homogeneity both imply quasi-homogeneity, which itself implies almost-homogeneity. If $\O$ is proper, then it is almost-homogeneous if, and only if, for all $a \in \partial \O$ there exist $(g_k) \in \Conf(\O)^{\mathbf{N}}$ and $x \in \O$ such that $g_k \cdot x \rightarrow a$ (see Remark~\ref{prop_asymp_dist_mark} below).

\subsection{Dual convexity}
\label{sectiondualconvex}

A strong feature of the theory of proper almost-homogeneous domains of flag manifolds is the notion of dual convexity. 
\begin{definition}
    We say that a domain $\O\subset\Ein^{p,q}$ is \emph{dually convex} if, for every $a\in \partial{\O}$ there is a point $b\in \Ein^{p,q}$ such that $a\in C(b)$ and $C(b)\cap \O=\emptyset$.
    In that case we say that $C(b)$ is a \emph{supporting lightcone} of $\O$ at $a$.
\end{definition}
  This notion was first introduced by A. Zimmer \cite{Zimpropqh} for arbitrary flag manifolds. The author proves that for a proper domain $\O$, dual convexity is equivalent to the completeness of an invariant distance function on $\O$: the \emph{Caratheodory metric $C_\O$}. Since the conformal group of $\O$ acts isometrically on $(\O,C_\O)$, any proper quasi-homogeneous domain is a complete metric space and hence it is dually convex. 

\begin{fact}[\cite{Zimpropqh}]
    Let $\O$ be a proper quasi-homogeneous domain of $\Ein^{p,q}$. Then $\O$ is dually convex.
\end{fact}

In this paper, we will need a slightly stronger result, whose proof relies on the ideas of~\cite{Zimpropqh}:  

\begin{prop}
    \label{propreetqhimpliquedc}
    Let $\O$ be a proper almost-homogeneous domain of $\Ein^{p,q}$. Then $\O$ is dually convex.
\end{prop}

\begin{proof}
    Let $\hat{\Omega}$ be the interior of the intersection of all proper dually convex domains containing $\Omega$. Then $\hat{\Omega}$ is a proper $\Conf(\Omega)$-invariant domain containing $\Omega$.
    In particular, the action of $\Conf(\Omega)$ on $\hat{\Omega}$ is proper. This follows from the fact that $\hat{\Omega}$ admits a conformally invariant distance (see Proposition \ref{prop_mark_hyp} or \cite{Zimpropqh}).
    Let us show that $\Omega$ is closed in $\hat{\Omega}$. Let $y\in \partial\Omega$. Since $\Omega$ is almost homogeneous, we can find $x\in \Omega$ and a sequence $(g_k)\in\Conf(\Omega)^\N$ such that $g_k(x)$ converges to $y$. The sequence $(g_k)$ must diverge, hence $g_k(x)$ converges to the boundary of $\hat{\Omega}$ since $\Conf(\Omega)$ acts properly on $\hat{\Omega}$. Therefore $y\in \partial\hat{\Omega}$, hence $\Omega$ is closed in $\hat \Omega$. Therefore $\Omega$ is a connected component of a dually convex domain, so $\Omega$ is also dually convex.  $\qedhere$
\end{proof}

This result and its proof can be found in \cite[Prop. 3.1.11]{galiaythesis} for domains in general flag manifolds.

\subsection{Diamonds} Let $p,q \geq 2$. We define \emph{diamonds}, which are models for the symmetric space of $\PO(p,1) \times \PO(1,q)$ in $\Ein^{p,q}$. Diamonds are proper symmetric domains, all conformally equivalent to each other, and by Theorem \ref{mainth}, they will be the only proper almost-homogeneous domain in $\Ein^{p,q}$. We start with the case where $p=1$, where diamonds are well known and admit an explicit description in terms of causality (see Section \ref{defdiamants} below). In Section \ref{sect_higher_diamonds}, we generalize the construction of diamonds to any signature.

\subsubsection{The Lorentzian diamond. }\label{defdiamants}  We write $\b(x,x)=-x_1^2+x_2^2+\dots+x_n^2$ for the bilinear form on $\mathbf{R}^{1,n-1}$. Given a point $x\in\R^{1,n-1}$, its future is the set $I^+(x)=\linebreak\{a\in\R^{1,n-1}\,\vert\,\b(a-x,a-x)<0\text{ and }\b(a-x,e_1)<0\}$. Similarly the past of $x$ is the set $I^-(x)=\{a\in\R^{1,n-1}\,\vert\,\b(a-x,a-x)<0\text{ and }\b(a-x,e_1)>0\}$. Now let $x,y\in\R^{1,n-1}$ such that $y\in I^+(x)$. The \emph{diamond defined by $x$ and $y$} is the proper domain, denoted by $\Diams(x,y)$, defined by  $$\Diams(x,y)=I^+(x)\cap I^-(y).$$
See Figure \ref{Diamant_et_D^pq}. More generally, let $x,y\in\Ein^{1, n-1}$ be two points such that $x\not\in C(y)$. Then $\Ein^{1, n-1}\smallsetminus(C(x)\cup C(y))$ has three connected components, two of which are proper and conformal to each other. These components are called the diamonds generated by $x$ and $y$. If $\Ein^{1,n-1}$ is endowed with a time orientation, we write $\Diams(x,y)$ for the connected component that contains a future directed curve joining $x$ to $y$. If $x$ and~$y$ belong to an affine chart where they are chronologically related, and if the induced orientation is the natural one on $\mathbf{R}^{1,n-1}$, we obtain the same definition as before. Note that $\PO(2,n)$ acts transitively on the set of diamonds, so the conformal structure of~$\Diams(x,y)$ is independent of the choice of $x$ and $y$. We write $\Diams^{1,n-1}$ for a model space with this conformal structure.

\begin{prop}
    $\Diams^{1,n-1}$ is conformally equivalent to $(\R\times\H^{n-1},[-dt^2\oplus g_{\H^{n-1}}]).$
\end{prop}

\begin{proof}
    Let $x,y\in\R^{1,n-1}$ such that $y\in I^+(x)$. A stereographic projection argument shows that $\Diams(x,y)$ is conformally equivalent to $I^+(x)\subset\mathbf{R}^{1,n-1}$ (\emph{send $y$ to infinity}). Fix an origin $0\in\R^{1,n-1}$ and identify $\H^{n-1}$  with $\{z\in I^+(0)\,\vert\,\b(z,z)=-1\}$. Then the map $\varphi:\R\times\H^{n-1}\to I^+(x)$ given by $\varphi(t,z)=e^t(x+z)$ is a conformal diffeomorphism.$\qedhere$
\end{proof}

From this identification, one can check using Liouville Theorem that every diamond is homogeneous, with conformal group
$$\Conf(\Diams^{1,n-1})\simeq \Iso(\R)\times\Iso(\H^{n-1})\simeq(\Z_2\ltimes\R)\times \PO(1,n-1).$$
In particular, the domain $\Diams^{1,n-1}$ is divisible and admits quotients of the form $\S^1\times \Sigma$, where $\Sigma$ is a compact Riemannian manifold of constant negative curvature.

\subsubsection{Diamonds and other homogeneous domains of $\Ein^{p,q}$. }\label{sect_higher_diamonds} We first introduce models for constant curvature geometry in higher signature. We let $V$ be a finite-dimensional real vector space and $\b$ be a quadratic form of signature $(p,q)$ on $V$. The sheet $\H^{p-1,q}\subset V$ is defined by 
$$\H^{p-1,q}=\left\{x\in V\,\vert\,\b(x,x)=-1\right\}.$$
The metric $\b$ restricts to a complete pseudo-Riemannian metric of signature $(p-\nolinebreak 1,q)$ and of constant negative curvature on $\H^{p-1,q}$. The two connected components of $\H^{0,n}$ are standard models for the real $n$-dimensional hyperbolic space, and we denote them simply by $\H^n$. The space $\H^{1,n}$ is usually referred to as the \emph{anti de Sitter space} in Lorentzian geometry. Similarly we define 
$$\dS^{p,q-1}=\left\{x\in V\,\vert\,\b(x,x)=+1\right\},$$
so that the metric $\b$ restricts to a complete pseudo-Riemannian metric of signature $(p,q-1)$ and of constant positive curvature on $\dS^{p,q-1}$. The space $\dS^{0,n}$ is a model for the round sphere, we will denote it by $\S^n$. The space $\dS^{1,n}$ is called the \emph{de Sitter space} in Lorentzian geometry. We will use the notation $-\H^n$ to denote $\dS^{n,0}$.

Let $\mathbf{R}^{p+1,q+1}=V_\pp \oplus V_\mm$ be an orthogonal decomposition. For $i \in \{ \pp, \mm\}$, we write $(p_i,q_i)$ for the signature of $V_i$ and we assume for instance that $p_\pp \leq p_\mm$. We denote by $F_i$ the intersection $F_i=\Ein^{p,q}\cap \mathbf{P}(V_i)$ (this intersection might be empty). Let~$J\subset \Ein^{p,q}$ be the \emph{joint} of $F_\pp$ and $F_\mm$, that is $J$ is the union of all photons intersecting  $F_\pp$ and~$F_\mm$. The domain $U=\Ein^{p,q}\smallsetminus J$ is the union of homogeneous connected components. More precisely we have the following description of $U$.

\begin{prop}
    \label{ouverts_homogènes_Einpq}
    One has the following 4 possible cases:
    \begin{itemize} 
        \item[*] If $p_i=0$ or $q_i=0$ for some $i\in\{\pp, \mm\}$, then the domain $U$ is connected, homogeneous (in fact symmetric) and dense in $\Ein^{p,q}$.
        \item[*] If $p_\pp=p_\mm=q_\pp=q_\mm=1$, then the domain $U$ has 4 connected components, all of which are Lorentzian diamonds.
        \item[*] If $p_\pp q_\mm=1$ and $p_\mm q_\pp\geq 2$, then the domain $U$ has 3 connected components, all of which are symmetric and two of which are proper and isomorphic to each other. The same conclusion holds if $p_\pp q_\mm\geq 2$ and $p_\mm q_\pp=1$.
        \item[*] If $p_\pp q_\mm\geq 2$ and $p_\mm q_\pp\geq 2$, then the domain $U$ has 2 connected components, which are both symmetric and nonproper.
    \end{itemize}
    
\end{prop}

\begin{proof}
Assume first that $p_i, q_i \geq 1$ for $i \in \{\pp, \mm\}$. One can write explicitly 
$$F_{i} = \mathbf{P}\left\{v_{i} \in V_{i} \, \vert \, v_{i} \neq 0, \ \b(v_i, v_{i}) = 0 \right\},$$
such that 
$J = \mathbf{P}\left\{v_\pp + v_\mm \in V_\pp \oplus V_\mm \, \vert \, \b(v_{i}, v_i) = 0 \text{ for } i \in \{\pp, \mm\}\right\}$.
Therefore, we can express $U$ as the union $U = U_{\pp} \sqcup U_{\mm}$, where 
$$U_{i} = \mathbf{P}\left\{v_\pp + v_\mm \in V_\pp \oplus V_\mm \, \vert \, -\b( v_\pp, v_\pp) = \b(v_\mm, v_\mm) = (-1)^i \right\}.$$ 
Now the map 
$$
\begin{array}{cccl}
         \pi: & \H^{p_\pp-1,q_\pp} \times \dS^{p_\mm,q_\mm-1} & \longrightarrow & U_\pp \\
         & (v_\pp, v_\mm) & \longmapsto & \mathbf{P}(v_\pp + v_\mm).
\end{array}
$$
defines a conformal 2-sheeted covering. The nontrivial deck transformation $\varphi$ centralizes $\PO(p+1, q+1)$, so $U_\pp \simeq \H^{p_\pp-1,q_\pp} \times \dS^{p_\mm,q_\mm-1}/_{x \sim \varphi(x)}$ is symmetric with isometry group $\PO(p_\pp-1, q_\pp) \times \PO(p_\mm, q_\mm-1)$. If $p_\pp \neq 1$ or $q_\mm \neq 1$, then at least one of the factors $\H^{p_\pp-1,q_\pp}$ or $\dS^{p_\mm,q_\mm-1}$ contains a lightlike geodesic $\gamma$ defined over $\R$. To find such a geodesic, intersect any degenerate plane of signature $(0,-)$ or $(0,+)$ with $\H^{p_\pp-1,q_\pp}$ or $\dS^{p_\mm,q_\mm -1}$, respectively. In particular, the domain $U_\pp$ contains a photon minus a point. Hence $U_\pp$ is neither Markowitz-hyperbolic nor proper. If $p_\pp = q_\mm = 1$, the total space $\H^{0,q_\pp} \times \dS^{p_\mm,0}$ has 4 connected components and $U_{\pp}$ is the union of two connected components both conformal to $(-\H^{q_\pp}) \times \H^{p_\mm}$. In order to write the components of $U_{\pp}$ explicitly, let $e_\pp\in \H^{0,q_\pp}$ and $e_\mm\in\dS^{p_\mm,0}$. Then 
$$U_{\pp}=D_\pp\sqcup D_\mm ,$$ 
where 
$$D_{i}=\P\{v_\pp+v_\mm\in U_\pp\,\vert\,(-1)^i\b(v_\pp,e_\pp)\b(v_\mm,e_\mm)>0\}.$$
For $x=\P(v_\pp  +v_\mm )\in D_\pp$ and $y=\P(w_\pp +w_\mm)\in D_\mm$, the signs of $\b(v_\pp, w_\pp)$ and $\b(v_\mm,w_\mm)$ are the same. In particular the value of $\b(u,w)$ cannot be zero. This means that the lightcone of every element of $D_\pp$ (resp.\ $D_\mm$) does not intersect $D_\mm$ (resp.\ $D_\pp$). In particular these components are proper. Similarly, $U_\mm$ is symmetric and has one or two connected components depending on the values of $p_\mm$ and $q_\pp$. In the case where (for instance) $p_\pp = 0$, the subset $F_\pp$ is empty, and $J = F_\mm$ has an empty interior in~$\Ein^{p,q}$. Thus $U$ is dense in $\Ein^{p,q}$. A map similar to $\pi$ shows that $U$ is connected and symmetric.$\qedhere$
\end{proof}

The proper components mentioned in Proposition~\ref{ouverts_homogènes_Einpq} are called \emph{diamonds}, we will denote them by $\Diams^{p,q}$. In the proof of Proposition~\ref{ouverts_homogènes_Einpq}, one gets the following conformal identification
$$\Diams^{p,q} \simeq (\H^p \times \H^q,[-g_{\H^p} \oplus g_{\H^q}]).$$
We now describe a concrete way to obtain a diamond as a bounded domain of $\R^{p,q}$. Let~$H_p$ be a negative definite affine $p$-plane in $\R^{p,q}$, and let $H_q$ be a positive definite affine $q$-plane orthogonal to $H_p$. Let $c$ be the intersection of $H_p$ with $H_q$, and let $r>0$. The flat metric restricts to a negative definite product on $H_p$ and a positive definite product on $H_q$. We write $S_{p-1}$ and $S_{q-1}$ for the balls of center $c$ and radius $r$ in $H_p$ and~$H_q$, respectively. Given a point $a \in S^{p-1}$ and $b \in S^{q-1}$, then the affine segment~$[a,b]$ is lightlike. Then the union $\mathcal{S}$ of all such segments is a topological sphere (it is a topological joint of $S_{p-1}$ and $S_{q-1}$). It separates $\R^{p,q}$ into two connected components, one of which is bounded and convex. This component is a diamond. Write $D^{p,q}$ for this component. We define for $x \in \R^{p,q}$ 
$$\Vert x \Vert_{p,q} = \sqrt{-\b(x_p,x_p)} + \sqrt{\b(x_q,x_q)},$$
where $x = x_p + x_q \in \overset{\rightarrow}{H_p} \oplus \overset{\rightarrow}{H_q}$. Then $\Vert \cdot \Vert_{p,q}$ defines a norm on $\R^{p,q}$ which depends on~$H_p$ and $H_q$, and we can describe $D^{p,q}$ as a ball for this norm: 
$$D^{p,q} = \{x \in \R^{p,q} \, \vert \, \Vert x-c \Vert_{p,q} < r\}.$$
In particular, the domain $D^{p,q}$ is a convex domain of $\R^{p,q}$. See Figure \ref{Diamant_et_D^pq}.
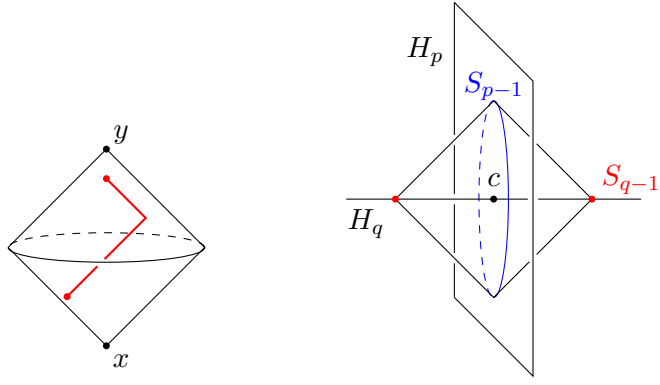
\begin{figure}
    \centering
    $$\begin{array}{cccccc}
        \begin{tikzpicture}[scale=1.3]
        \draw[dashed] (1,0) arc (0:180:1 and 0.15);
        \draw[red,thick] (-0.4,-0.5) -- (0.4,0.3) -- (0,0.7);
          \fill[fill=white] (-0.06,-0.15) circle (1.5pt);
          \draw (-1,0) -- (0,-1) -- (1,0) -- (0,1) -- (-1,0);
          \draw (-1,0) arc (180:360:1 and 0.15);
          \fill[fill=black] (0,1) circle (1pt);
          \fill[fill=black] (0,-1) circle (1pt);
          \fill[fill=red] (-0.4,-0.5) circle (1pt);
          \fill[fill=red] (0,0.7) circle (1pt);
          \draw node at (0.15,1.15) {$y$};
          \draw node at (0.15,-1.15) {$x$};
        \end{tikzpicture}
        & & & & &
    \begin{tikzpicture}[scale=1.3]
        \draw (-0.4,+1.6+0.4) -- (-0.4,-1.4+0.4);
        \fill[fill=white] (-0.4,+0.6) circle (1.5pt);
        \fill[fill=white] (-0.4,-0.6) circle (1.5pt);
        \fill[fill=white] (-0.4,0) circle (1.5pt);
        \begin{scope}[rotate = 90,blue]
          \draw (-1,0) arc (180:360:1 and 0.15);
        \end{scope}
        \begin{scope}[rotate = 90,blue]
          \draw[dashed] (1,0) arc (0:180:1 and 0.15);
        \end{scope}
        \draw[blue] node at (0,1.15) {$S_{p-1}$};
        \draw (-1,0) -- (0,-1) -- (1,0) -- (0,1) -- (-1,0);
        \draw (-1.5,0) -- (1.5,0);
          \fill[fill=white] (0.4,0) circle (1.5pt);
        
          \fill[fill=white] (0.4,+0.6) circle (1.5pt);
          \fill[fill=white] (0.4,-0.6) circle (1.5pt);
          \draw (-0.4,-1.4+0.4) -- (+0.4,-1.4-0.4) -- (+0.4,+1.6-0.4) -- (-0.4,+1.6+0.4);

          \fill[red] (1,0) circle (1pt);
          \fill[red] (-1,0) circle (1pt);
          \draw[red] node at (1.4,0.2) {$S_{q-1}$};

          \draw node at (-1.3,-0.25) {$H_q$};
          \draw node at (-0.7,1.5) {$H_p$};
          
          \fill[fill=black] (0,0) circle (1pt);
          \draw node at (0,0.2) {$c$};
        \end{tikzpicture}
    
    \end{array}$$
    \caption{On the left: the Lorentzian diamond in $\R^{1,2}$. Two points in the diamond are joined by a sequence of two segments of photons (see Fact \ref{pathsindiamonds}). On the right: a diamond in $\R^{p,q}$. The spheres $S_{p-1}$ and $S_{q-1}$ correspond to the subsets $F_\pp$ and $F_\mm$ in Proposition~\ref{ouverts_homogènes_Einpq}. The boundary $\mathcal{S}$ of $D^{p,q}$ is a subset of the set $J$ in Proposition~\ref{ouverts_homogènes_Einpq}.}
    \label{Diamant_et_D^pq}
\end{figure}

\begin{rmk}
    There is a purely causal way to define $D^{p,q}$. In signature $(p,q)$, one can define the future of an inextensible $(p-1)$-timelike curve (see \cite{Romeo}). In that setting, the diamond is the future of the timelike sphere $S^{p-1}$.
\end{rmk}

We will use the following fact later:
\begin{fact} \label{pathsindiamonds}
    Any two points $x,y \in \Diams^{p,q}$ can be joined by a sequence of two segments of photons contained in $\Diams^{p,q}$. See Figure \ref{Diamant_et_D^pq}.
\end{fact}

\begin{proof}
    Write $x=(x_p,x_q)$ and $y=(y_p,y_q)$ in the model $\H^p\times\H^q$. We let $d$ refer to the hyperbolic distance function. Assume for instance that $d(x_p,y_p)<d(x_q,y_q)$. Let~$z_p\in \H^p$ such that $d(x_p,z_p)+ d(x_p,y_p) = d(z_p,y_p)$ and let $z_q\in [x_q,y_q]\subset\H^q$ such that $d(x_q,z_q)=d(x_p,z_p)$. Then the lightcone of $z=(z_p,z_q)$ contains both $x$ and $y$.$\qedhere$
\end{proof}

\section{The Markowitz distance function}\label{kobayashimetric}

In this section we define the Markowitz pseudodistance, and investigate its properties. This pseudodistance makes sense for an arbitrary conformal manifold, but we will only define it for domains of $\Ein^{p,q}$. See \cite{markowitz_1981} for the definition and study of that pseudodistance in the general case, and see \cite{dist_mark} for a recent overview on that subject.

\subsection{Definition of the pseudodistance and Markowitz hyperbolicity} Let $\O$ be a domain of $\Ein^{p,q}$ (not necessarily proper). Let $I=(-1,1)$ be the one-dimensional hyperbolic space endowed with its projectively invariant hyperbolic metric~$g_{hyp}=\frac{4dx^2}{(1-x^2)^2}$. We write $d_{hyp}$ for the associated hyperbolic distance on $I$. Let $\Proj(I, \O) \subset \Proj(I,\mathbf{RP}^{p+q+1})$ be the set of projective immersions $\alpha:I\to \mathbf{RP}^{p+q+1}$ whose image lies entirely in $\O$. The image of a curve $\alpha\in \Proj(I, \O)$ is a conformal lightlike geodesic. In fact the curve $\alpha$ is a projectively parametrized lightlike geodesic in the sense of \cite{markowitz_1981} and every projectively parametrized lightlike geodesic is (locally) equal to an element of $\Proj(I, \O)$. This allows us to define the Markowitz pseudodistance in the following way.
Let $x,y \in \O$. A \emph{chain from $x$ to $y$ (of index $N\in\N$)} is a sequence of points $x_0=x, x_1,\dots, x_N=y\in \O $  together with sequences $s_1,\dots,s_{N},t_1,\dots,t_{N}\in I$ and projectively parametrized lightlike geodesics $\alpha_i:I\to\O$ such that $\alpha_i(s_i)=x_{i-1}$ and $\alpha_i(t_i)=x_{i}$. We will write it $((x_0, \cdots, x_{N}), (\alpha_i)_{1 \leq i \leq N-1})$. Then we define 
$$\delta^{\O}(x,y)=\inf\sum_{i} d_{hyp} (s_i , t_i),$$
where the infinimum is taken over all chains from $x$ to $y$. This definition coincides with the definition of Markowitz for pseudo-Riemannian manifolds (see \cite{markowitz_1981}). 

Since a conformal map $f$ sends a chain from $x$ to $y$ to a chain from $f(x)$ to $f(y)$, we have the following standard and important naturality properties.

\begin{prop} 
    \begin{enumerate}\label{propdebasedmark}
    \item Given two domains $\O_1, \O_2$ of $\Ein^{p,q}$ such that $\O_1 \subset \O_2$, one has: $\delta^{\O_2}(x,y) \leq \delta^{\O_1}(x,y)$ for all $x,y \in \O_1$.
    \item For every domain $\O\subset\Ein^{p,q}$, the pseudodistance $\delta^{\O}$ is $\Conf(\O)$-invariant.
\end{enumerate}
\end{prop}

Note that $\delta^{\O_2}$ and $\delta^{\O_1}$ and $\delta^{\O}$ do not need to be distance functions in Proposition~\ref{propdebasedmark}.
Given a domain $\O$ and points $x,y \in \O$, we write $\mathcal{C}_{x,y}$ (resp.\ $\mathcal{C}^N_{x,y}$) for the collection of all chains from $x$ to $y$ (resp.\ chains of index lower or equal to $N$). For $x,y$ sufficiently close to each other, we can find a diamond included in $\O$ and containing $x$ and $y$. In that case $\mathcal{C}_{x,y}$ is nonempty (and even $\mathcal{C}_{x,y}^2$ is nonempty by Fact~\ref{pathsindiamonds}). Since the relation ``$x$ and $y$ can be joined by a chain'' is an equivalence relation, the set $\mathcal{C}_{x,y}$ is never empty. Therefore $\delta^{\O}(x,y)$ is always finite. 

Because one can concatenate and reverse the orientation of a chain, the \linebreak map $\delta^{\O}:\O\times\O\to\R$ is symmetric and satisfies the triangular inequality. Therefore $\delta^{\O}$ is always a pseudodistance. A domain $\O$ is said to be \emph{Markowitz-hyperbolic} if~$\delta^\O$ is a distance function, that is, if $\delta^\O$ separates points. 

\begin{ex}
\label{distance_du_diams}
    Diamonds are Markowitz-hyperbolic domains. In the model $\Diams^{p,q}=(-\H^p)\times\H^q$, one has $$\delta^{\Diams^{p,q}}((x_p,x_q),(y_p,y_q))=\max\{d(x_p,y_p),d(x_q,y_q)\},$$ 
    where $d(x,y)$ is the hyperbolic distance between $x$ and $y$. Write $g$ for the Riemannian metric $g_{\H^p}\oplus g_{\H^q}$. We use of the Markowitz \emph{infinitesimal conformal pseudometric} $F_\Mark$ (see \cite{markowitz_1981}). For a lightlike vector $X\in T_x\Diams^{p,q}$, Markowitz defines 
    $F_\Mark(X)=\inf_\alpha \vert u\vert$, where the infinimum is taken over all projectively parametrized photons $\alpha:I\to\Diams^{p,q}$ and $u\in T_0I$ such that  $\alpha_*(u)=X$.
    Let $\mathbf{L}{\Diams^{p,q}}$ be the bundle of lightlike directions of ${\Diams^{p,q}}$. Since ${\Diams^{p,q}}$ is homogeneous and the stabilizer $H_x$ of a point $x\in{\Diams^{p,q}}$ sends any lightlike line in $T_x\D^{p,q}$ to any other lightlike line in $T_x\D^{p,q}$, the space $\mathbf{L}\Diams^{p,q}$ is $\Conf({\Diams^{p,q}})$-homogeneous. Let $X_0\in T{\Diams^{p,q}}$ be a nonzero lightlike vector and let $\alpha\geq 0$ such that $F_\Mark(X_0)=\alpha\sqrt{g(X_0,X_0)}$. Since $F_\Mark$ is a conformally invariant functional, the same equality holds for any vector in the $\Conf(\D^{p,q})$-orbit of $X_0$. Because $F_\Mark$ satisfies $F_\Mark(\lambda X)=\lambda F_\Mark(X)$ for every $\lambda>0$ and because $\mathbf{L}\Diams^{p,q}$ is homogeneous, we get that for every lightlike vector $X\in T\D^{p,q}$
    $$F_\Mark(X)=\alpha\sqrt{g(X,X)}.$$
    Since $\R^{p,q}$ is Ricci flat, every affine parameter for geodesics of $\R^{p,q}$ is a projective parameter in the sense of \cite{markowitz_1981}. A direct computation, using any maximal segment $\ell\subset \D^{p,q}\subset \R^{p,q}$, shows that $\alpha=1/\sqrt{2}$. From Theorem 4.8 of \cite{markowitz_1981}, we can conclude that for $x,y\in\H^p\times \H^q$ $$\delta^{\Diams^{p,q}}(x,y)=\frac{1}{\sqrt{2}}\inf_\gamma L(\gamma,g),$$
    where the infinimum is taken over every piecewise Riemannian geodesics $\gamma=(\gamma_p,\gamma_q)$ of $\H^p\times \H^q$ from $x$ to $y$ such that $\gamma_p$ and $\gamma_q$ are parametrized by arc length. The result follows.
\end{ex}

Once equipped with the distance function $\delta^{\Diams^{p,q}}$, the diamond $\Diams^{p,q}$ is a homogeneous locally compact metric space and is therefore a proper and geodesic metric space (see \cite[Prop.\ 3.7]{bridson2013metric}). The topology induced by $\delta^{\Diams^{p,q}}$ is the standard one. In fact for a general proper dually convex domain $\O$, the pseudodistance $\delta^\O$ will be a proper and geodesic distance function defining the standard topology of $\O$.

\begin{prop}\label{prop_mark_hyp}
    Proper domains of $\Ein^{p,q}$ are Markowitz-hyperbolic.
\end{prop}

\begin{proof}
    By Proposition~\ref{propdebasedmark}, any open subset of a Markowitz-hyperbolic domain is itself hyperbolic. Since any proper domain of $\Ein^{p,q}$ is contained in a diamond, it is Markowitz-hyperbolic.$\qedhere$
\end{proof}

\begin{rmk}
    \label{Ouvert_MH_homogene}
    A Markowitz-hyperbolic domain needs not be proper. For instance, the Misner domain $M$ defined as the future of two transverse degenerate hyperplanes in $\mathbf{R}^{1,n-1}$ is an example of a Markowitz-hyperbolic homogeneous domain which is not proper, as $\partial M$ contains infinitely many photons.
\end{rmk}

\begin{prop}\label{distmark}
    Let $\O \subset \Ein^{p,q}$ be a proper domain.  Then $\delta^{\O}$ induces the standard topology on $\O$. 
\end{prop}

\begin{proof}
We need to show that $\delta^\O$ defines the same convergence as the one induced on~$\O$ by the topology of $\Ein^{p,q}$. Let $(x_k)$ be a sequence of $\O$ and let $x\in \O$. Assume first that $x_k\to x$ as $k\to \infty$. Let  $D$ be a diamond containing $x$ and contained in $\O$. Example \ref{distance_du_diams} shows that $\delta^D$ defines the standard topology on $D$, so that $\delta^D (x, x_k) \to 0$ as $k\to+\infty$.
By Proposition~\ref{propdebasedmark}, one has $\delta^{\O}(x,x_k) \leq \delta^D (x, x_k)$. Thus $\delta^\O(x,x_k)\to 0$ as $k\to \infty$. 

The converse implication is immediate if $\O$ is proper. Indeed assume that $\delta^\O(x,x_k)\to 0$ as $k\to \infty$ and let $D$ be a diamond containing $\O$. Then $\delta^D(x,x_k) \leq \delta^\O (x, x_k)$ so that $\delta^D(x,x_k)\to 0$ as $k\to \infty$. By Example \ref{distance_du_diams} this implies that $(x_k)$ converges to $x$. 
$\qedhere$    
\end{proof}
\begin{rmk}\label{prop_asymp_dist_mark}
    If $\O$ is proper and $(x_k), (y_k) \in \O^{\mathbf{{N}}}$ satisfy $\delta^{\O}(x_k, y_k) \rightarrow 0$, and if $x_k \rightarrow x \in \partial \O$, then $y_k \rightarrow x$. To prove this, take a diamond $D$ containing $\overline{\O}$, and apply the proof of Proposition~\ref{distmark}. 
    By invariance of the Markowitz distance function, this fact implies that if $\O \subset \Ein^{p,q}$ is proper and almost-homogeneous, then for any $a \in \partial \O$, there exist $(g_n) \in \Conf(\O)^{\mathbf{N}}$ and $y \in \O$ such that $g_n \cdot y \rightarrow a$.
\end{rmk}

\subsection{Properness of the Markowitz distance function}
\label{properness}

Let $\O$ be a Markowitz-hyperbolic domain and let $x, y \in \O$. Any sequence
$$((x_0, \cdots, x_{N}), (\alpha_i)_{1 \leq i \leq N-1}) \in \mathcal{C}_{x,y}$$
with $(s_i), (t_i) \in I$ such that $\alpha_i (s_i) = x_{i-1}$, $\alpha_i (t_i) = x_i$, gives rise to a continuous path $\gamma$ from $x$ to $y$, with image the union of all the images of $\alpha_i |_{[s_i,t_i]}$. This path is unique up to reparametrization.
We write $L(\gamma)$ for the length of $\gamma$ measured with the metric $\delta^{\O}$. This length does not depend on the choice of parametrization. Hence we will identify the elements of $\mathcal{C}_{x,y}$ with continuous paths $\gamma$ arising from this construction, by means of an arbitrary choice of parametrization.

\begin{lem}\label{length metric}
    Assume $\O$ is a proper dually convex domain of $\Ein^{p,q}$. Then $\delta^{\O}$ is a length metric on $\O$. Moreover, for any maximal lightlike geodesic $\alpha: I \rightarrow \O$ and any $s,t \in I$, the length of the segment $\alpha_{[s,t]}$ is equal to $d_{hyp}(s,t)$. 
\end{lem}
\begin{proof} Let us fix some notation. Given two points $x,y \in \Ein^{p,q}$, there exist $u,v \in \mathbf{R}^{p+q+2} \smallsetminus \{0\}$ such that $x = \mathbf{P}(u)$ and $y = \mathbf{P}(v)$. Given two points $\xi_1 = \mathbf{P}(v_1), \xi_2 = \mathbf{P}(v_2) \in \Ein^{p,q}$, each quantity $\langle \xi_i, x\rangle$, $\langle \xi_i, y\rangle$ is not well defined in itself, however, the quantity 

\begin{equation*}
    |[\xi_1: x: y:  \xi_2]| := \left| \frac{\langle v_1, u \rangle \cdot \langle v_2, v \rangle}{\langle v_2, u\rangle \cdot \langle v_1, v\rangle} \right|
\end{equation*}
makes sense as soon as $\xi_1, \xi_2 \notin C(x) \cup C(y)$, because it does not depend on the choice of generators of $x,y, \xi_1$ and $\xi_2$. Moreover, we will write $\langle \xi, x \rangle \ne 0$ if, and only if $\langle v_1, u \rangle \ne 0$. This property does not depend on the choice of generators. \\
\indent Let $\mathbf{R}^{p,q}$ be some affine chart containing $\overline{\O}$. Let us show that maximal lightlike geodesics (i.e.\ lightlike geodesics $\alpha: I \rightarrow \O$ with endpoints in $\partial \O$) in $\O$ are geodesics for $\delta^{\O}$. Let us define
\begin{equation*}
    K(\O) := \{ \xi \in \Ein^{p,q} \mid \langle \xi, x \rangle \ne 0 \quad \forall x \in \O\}.
\end{equation*}
Let $\alpha: I \rightarrow \O$ be a maximal lightlike geodesic. Let $\Delta$ be the photon of $\Ein^{p,q}$ such that the image of $\alpha$ is contained in $\Delta$. Let $u \in \mathbf{R}^{p+q+2} \smallsetminus \{0\}$ such that $\alpha(0) = \mathbf{P}(u)$. Let $\{x_{\infty} \} = \Delta \smallsetminus \mathbf{R}^{p,q} $ and $v \in \mathbf{R}^{p+q+2} \smallsetminus \{0\}$ such that $x_{\infty} = \mathbf{P}(v)$. Let $a,b$ be the intersection of $\Delta$ with $\partial \O$, so that $a$ and $b$ are the endpoints of $\alpha$. Then there exist $\lambda, \mu \in \mathbf{R}$ such that $a = \mathbf{P}(u + \lambda v)$ and $b = \mathbf{P}(u + \mu v)$. By dual convexity, there exist $\xi_1 = [v_1], \xi_2 = \mathbf{P}(v_2) \in K(\O)$ such that $a \in C(\xi_1)$ and $b \in C(\xi_2)$. Then
\begin{equation*}
    \lambda = - \frac{\langle v_1, u \rangle }{\langle v_1, v \rangle} \text{ and } \mu = - \frac{\langle v_2, u \rangle }{\langle v_2, v \rangle}.
\end{equation*}
Now let $t \in I \smallsetminus \{ 0 \}$ and let $\mu' \in \mathbf{R}^*$ such that $\alpha(t) = \operatorname{Span}(u + \mu' v)$. Then 
\begin{equation*}
    d_{hyp}(0,t) = \log \left| \frac{\lambda \cdot (\mu-\mu') }{\mu \cdot (\lambda- \mu')} \right| = \log \left| \frac{\langle v_1, u \rangle \cdot \langle v_2, u + \mu' v \rangle}{\langle v_2, u  \rangle \cdot \langle v_1, u+ \mu' v \rangle} \right| = \log|[\xi_1: \alpha(0): \alpha(t):  \xi_2]|.
\end{equation*}
Then $d_{hyp}(0,t) \leq \sup_{\xi_1, \xi_2 \in K(\O)} \log|[\xi_1: \alpha(0): \alpha(t):  \xi_2]|$.

For the converse inequality, let $\xi_1, \xi_2 \in K(\O)$ and let $a',b' \in \Delta $ such that $\{a'\} = \Delta \cap C(\xi_1)$ and $\{b'\} = \Delta \cap C(\xi_2)$. Let $\beta: I \rightarrow \Ein^{p,q}$ be the lightlike geodesic with endpoints $a'$ and $b'$, whose image contains the image of $\alpha$. Then there exist $s,t' \in I$ such that $\beta(t') = \alpha(t)$ and $\beta(s) = \alpha(0)$. The same argument  
\begin{equation*}
d_{hyp}(s,t') = \log |[\xi_1: \beta(s): \beta(t'):  \xi_2]| =  \log |[\xi_1: \alpha(0): \alpha(t):  \xi_2]|
\end{equation*}
But $a',b' \notin \O$, so $d_{hyp}(s,t') \leq d_{hyp}(0, t)$. Hence
\begin{equation*}
        d_{hyp}(0,t) \geq \log|[\xi_1: \alpha(0): \alpha(t):  \xi_2]|.
\end{equation*}
This is true for all $\xi_1, \xi_2 \in K(\O)$, so $d_{hyp}(0,t) = \sup_{\xi_1, \xi_2 \in K(\O)} \log|[\xi_1: \alpha(0): \alpha(t):  \xi_2]|$.

Similarly, for all $s,t \in I$, one has
\begin{equation}\label{egalitédiqtances}
    d_{hyp}(s,t) = \sup_{\xi_1, \xi_2 \in K(\O)} \log|[\xi_1: \alpha(s): \alpha(t):  \xi_2]|.
\end{equation}
Now fix $s,t \in I$ and $x = \alpha(s)$, $y = \alpha(t)$. Let $((x_0, \cdots, x_{N}), (\alpha_i)_{1 \leq i \leq N-1}) \in \mathcal{C}_{x,y}$.
Then by~\eqref{egalitédiqtances} one has
\begin{equation*}
\begin{split}
    \sum_i d_{hyp}(s_i, t_i) &\geq \sum_i \sup_{\xi_1, \xi_2 \in K(\O)} \log |[\xi_1: x_i: x_{i+1}: \xi_2]| \\
    &\geq \sup_{\xi_1, \xi_2 \in K(\O)} \sum_i  \log |[\xi_1: x_i: x_{i+1}: \xi_2]|  \\
    & = \sup_{\xi_1, \xi_2 \in K(\O)}  \log |[\xi_1: x: y: \xi_2]|= d_{hyp}(s,t).
\end{split}
\end{equation*}
This is true for all $((x_0, \cdots, x_{N}), (\alpha_i)_{1 \leq i \leq N-1}) \in \mathcal{C}_{x,y}$, so $\delta^{\O}(x,y) \geq d_{hyp}(s,t)$. The converse inequality is straightforward. 

In particular, the length of a chain $\gamma = ((x_0, \cdots, x_{N}), (\alpha_i)_{1 \leq i \leq N-1}) \in \mathcal{C}_{x,y}$ is 
\begin{equation*}
    L(\gamma) = \sum_i \delta^{\O}(x_i, x_{i+1}).
\end{equation*}
Then one has $\delta^{\O}(x,y) = \inf_{\gamma \in \mathcal{C}_{x,y}} L(\gamma)$. Now let $\mathcal{P}_{x,y}$ the set of all rectifiable curves joining $x$ and $y$ in $\O$. By definition of the length of a curve,
$$\inf_{\gamma \in \mathcal{C}_{x,y}} L(\gamma) = \delta^{\O}(x,y) \leq \inf_{\gamma \in \mathcal{P}_{x,y}} L(\gamma).$$
Since chains are rectifiable (for the above identification with continuous paths), this inequality is an equality. Hence $\delta^{\O}$ is a length metric. $\qedhere$
\end{proof}

\begin{cor}
\label{dpropgeod}
Let $\O$ be a Markowitz-hyperbolic almost-homogeneous domain of $\Ein^{p,q}$. Then $\delta^{\O}$ is a proper distance function.
\end{cor}

\begin{proof}
By Propositions \ref{propdebasedmark} and \ref{distmark} and Lemma~\ref{length metric}, the space $(\O, \delta^{\O})$ is a locally compact length metric space and since $\Conf(\O)\subset\Iso (\O, \delta^{\O})$ there exists some compact subset $\mathcal{K} \subset \O$ such that $ \O = \Iso (\O, \delta^{\O}) \cdot \mathcal{K}$. This implies that $(\O, \delta^{\O})$ is complete (see \cite[Lem.\ 9.2]{Zimpropqh}). Then by Hopf--Rinow Theorem, the distance function $\delta^{\O}$ is proper (see \cite[Prop.\ 3.7]{bridson2013metric}). $\qedhere$
\end{proof}

Note that Hopf--Rinow Theorem also implies that the distance function $\delta^{\O}$ is geodesic, although we will not use this fact.

\begin{ex}\label{exnotcomplete}One can construct proper causally convex domains of $\mathbf{R}^{1,n-1}$ whose Markowitz distance function is incomplete. Let $D\subset \mathbf{R}^{1,n-1}$ be a diamond with past endpoint $x$ and let $y\in D$. The domain $\O=D\smallsetminus J^+(y)$ is a causally convex domain contained in $D$ (see Figure \ref{exemple_non_complet}). Let $U=\Diams(x,y)$ be the shaded diamond in Figure \ref{exemple_non_complet}. We claim that $\delta^\O$ and $\delta^D$ coincide on $U$. We already have the inequality $\delta^D\leq\delta^\O$ on $U$. Let $a,b\in U$ and let $c\in U$ such that $[a,c]$ and $[b,c]$ are lightlike segments (exist by Fact~\ref{pathsindiamonds}). Let $\Delta$ (resp.\ $\Delta^\prime$) be the photon containing $a$ and $c$ (resp.\ $b$ and $c$) and let $I=\Delta\cap D$ (resp.\ $I=\Delta^\prime\cap D$). The union of $I$ and $I^\prime$ is a chain $C$ of $D$ from $a$ to $b$ such that $L(C)=\delta^D(a,b).$ Since both $I$ and $I^\prime$ are contained in $\Omega$, the chain $C$ is also a chain of $\Omega$ from $a$ to $b$. In particular $\delta^\Omega(a,b)\leq L(C)=\delta^D(a,b)$. This proves the converse inequality. Now let $(x_k)$ be a sequence of $U$ converging to $y\in \partial \Omega$. Then $(x_k)$ is a Cauchy sequence of $(\Omega,\delta_\Omega)$ that does not converge in $\Omega$.

\begin{figure}
    \centering
        \begin{tikzpicture}[scale=1.2]
        \draw[dashed] (2,2) arc (0:180:2 and 0.15);
        \fill[fill=white] (-0.13,2.13) circle (1.5pt);
        \fill[fill=white] (0.13,2.13) circle (1.5pt);
        \draw (0,0)--(2,2)--(1,3);
        \draw[dashed] (1,3)--(0,2)--(-1,3);
        \draw (-1,3)--(-2,2)--(0,0);
        \draw[dashed] (-0.5,2.5) arc (180:360:0.5 and 0.075);
        \draw[dashed] (0.5,2.5) arc (0:180:0.5 and 0.075);
        \draw (-1,3) arc (180:360:1 and 0.15);
        \draw (1,3) arc (0:180:1 and 0.15);
        \draw (-1,1) arc (180:360:1 and 0.15);
        \draw[dashed] (1,1) arc (0:180:1 and 0.15);
        \draw[fill, pattern = north east lines, opacity=0.5] (0,0)--(1,1)--(0,2)--(-1,1)--(0,0);
        \draw (1,1)--(0,2)--(-1,1);
        \fill[white] (-0.3,1.86-0.06)--(0.3,1.86-0.06)--(0.3,1.86+0.06)--(-0.3,1.86+0.06)--(-0.3,1.86-0.06);
        \draw (-2,2) arc (180:360:2 and 0.15);
        \draw node at (0.3,0) {$x$};
        \fill (0,0) circle (1.3pt);
        \draw node at (0.3,2) {$y$};
        \fill (0,2) circle (1.3pt);
        \draw node at (-0.7,0.3) {$U$};
    \end{tikzpicture}
    \caption{A proper causally convex domain with incomplete Markowitz distance function.}
    \label{exemple_non_complet}
\end{figure}
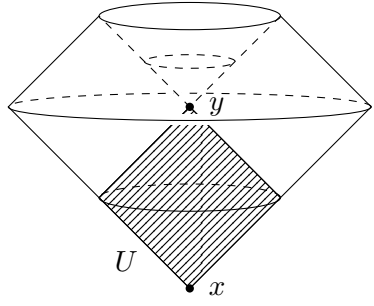
\end{ex}

\section{Dynamics at points of the boundary}\label{sect_dyn_boun}

When the conformal group of a proper domain of the Einstein universe is cocompact, it gives information on the structure of the domain and its boundary. In this section, using the Markowitz distance function, we investigate the structure of the boundary of a proper almost-homogeneous domain.

\subsection{Photon-extremal points}\label{sect_extr_are_visual}Let $\O \subset \Ein^{p,q}$ be a proper domain. Let $a \in \partial \O$. We say that $a$ is \emph{photon-extremal} if for any photon $\Delta$ passing through $a$, the point $a$ is not contained in $\operatorname{int}_\Delta (\Delta \cap \overline{\O})$, where $\operatorname{int}_\Delta(S)$ denotes the interior of a subset $S$ with respect to the induced topology on $\Delta$. We will denote by $\Extr(\O)$ the set of photon-extremal point of $\O$.

Our goal is to prove the following strong property of photon-extremal points, under the assumption that the domain is almost-homogeneous and proper. 

\begin{prop}
    \label{extremal_implique_conedisjoint}
    Let $\O$ be a proper almost-homogeneous domain. If $a\in\partial \O$ is photon-extremal, then $C(a)$ is disjoint from $\O$. 
\end{prop}

To prove Proposition~\ref{extremal_implique_visuel}, we need to introduce new functions of $\O \times \O$, related to the distance function $\delta^{\O}$. Let $x,y \in \O$, and $N \in \N^*$. Let us define
$$\delta^{\O}_N (x,y) := \inf_{\gamma \in \mathcal{C}_{x,y}^N} L(\gamma).$$
Note that this quantity can be infinite.
For every $x,y \in \O$, the sequence $(\delta^{\O}_N(x,y))_{N \in \N^*}$ is nonincreasing, eventually finite,  and one has $\delta^{\O}(x,y) = \lim_{N \rightarrow + \infty} \delta^{\O}_N(x,y)$. 

\begin{lem}
    \label{extremal_implique_visuel}
    Let $\O$ be a proper almost-homogeneous domain and $x_\infty\in\partial \O$ be photon-extremal. Then for all $(x_k), (y_k) \in \O^{\mathbf{N}}$ such that $x_k \rightarrow x_\infty$, if the sequence $(\delta^{\O}_N(x_k, y_k))$ is bounded for some $N \in \mathbf{N}$, then $y_k \rightarrow x_\infty$. 
\end{lem}

\begin{figure}[H]
    \centering
    \begin{tikzpicture}[scale = 3]

          \draw (-1,-1) -- (-0.8,-1.4) -- (-0.4,-1.85) -- (0,-2) -- (0.4,-1.85) -- (0.8,-1.4) -- (1,-1) ;
          \fill[white] (0.78,-1.42) circle (1pt);
          \draw node at (0.33,-1.8) {$y_\infty$};
          \draw node at (0.75,-1.4) {$x_\infty$};
          \draw node at (-0.05,-1.5) {$y_k$};
          \draw node at (0.23,-1.23) {$x_k$};
        \draw[red] (0.1,-1.9) -- (0.8,-1.2);
        \fill (0.3,-1.7) circle (0.7pt);
          \fill (0.6,-1.4) circle (0.7pt);
        \draw node at (0.19,-1.59) {$\ddots$};
        \draw node at (0.49,-1.29) {$\ddots$};
        \begin{scope}[xshift=-7,yshift=4.5]
            \draw[red] (0.1,-1.9) -- (0.8,-1.2);
        \fill (0.3,-1.7) circle (0.7pt);
          \fill (0.6,-1.4) circle (0.7pt);
        \end{scope}
        \fill[white] (0.47,-1.13) circle (1pt);
        \draw (1,-1) arc (0:360:1 and 0.15);
        \end{tikzpicture}

    \caption{}
    \label{fig_visual_points}
\end{figure}

\begin{proof}[Proof of Lemma~\ref{extremal_implique_visuel}]
    For all $k \in \N$, let $\gamma_k = ((x_0^k, \cdots , x_N^k), (\alpha_i^k)_i) \in \mathcal{C}_{x_k, y_k}^N$ such that  
    $$\sum_{i=0}^{N-1} \delta^{\O}_1(x_i^k, x_{i+1}^k) = L(\gamma_k) \leq \delta^{\O}_N (x_k, y_k) + 1 \leq M + 1.$$
    Then in particular, for all $i \in \{ 0, \dots, N-1\}$ one has $\delta^{\O}(x^k_i, x^k_{i+1}) \leq M + 1$. Hence one can assume that $N = 1$ and deduce the result by induction. Let us then assume that $N=1$. For all $k\in\N$, there exists a maximal lightlike geodesic $\alpha_k: I \rightarrow \O$ and $s_k, t_k \in I$ with $s_k < t_k$, such that $\alpha_k(s_k) = x_k$ and $\alpha_k(t_k) = y_k$. Then by Propositions~\ref{propreetqhimpliquedc} and~\ref{length metric}, one has
    
    $$\delta^{\O} (x_k, y_k) = \delta^{\O}_1 (x_k, y_k) = d_{hyp}(s_k, t_k),$$ 
    so the sequence $(d_{hyp}(s_k, t_k))$ is bounded.

 The image of $\alpha_k$ is contained in a unique photon $\Delta_k$. Let $a_k, b_k$ be the endpoints of $\Delta_k \cap \overline{\O}$, such that $a_k, x_k, y_k, b_k$ are aligned in this order. Up to extracting, we may assume that $y_k \rightarrow y_{\infty} \in \overline{\O}$, $a_k \rightarrow a_{\infty} \in \partial \O$ and $b_k \rightarrow b_{\infty} \in \partial \O$ as $k \rightarrow + \infty$. For all $k \in \N$, the points $a_k, x_k, y_k, b_k$ lie on the same photon $\Delta_k$, so $a_{\infty}, x_{\infty},y_{\infty}, b_{\infty}$ lie on the same photon $\Delta$, in this order. By extremality of $x_{\infty}$, we must have $x_{\infty} = a_{\infty}$. This implies that $t_k \rightarrow -1$. But then, since the sequence $(d_{hyp}(s_k, t_k))$ is bounded, we must also have $s_k \rightarrow -1$. Hence $y_{\infty} = x_{\infty}$. See also Figure \ref{fig_visual_points}.$\qedhere$
\end{proof}

We can now prove Proposition~\ref{extremal_implique_conedisjoint}:

\begin{proof}[Proof of Proposition~\ref{extremal_implique_conedisjoint}] Let $a\in\partial \O$ be photon-extremal. Since $\O$ is almost-homogeneous we can find $x\in \O$ and a sequence $(g_k)\in\Conf(\O)$ such that
    $$g_k\cdot x\underset{k\to+\infty}{\longrightarrow}a.$$
    Now let $y \in \O$ and $N \in \mathbf{N}$ such that $\delta^{\O}_N(x,y) < + \infty$. Then, by $\Conf(\O)$-invariance of $\delta^{\O}_N$, one has 
\begin{equation*}
    \delta^{\O}_N(g_k \cdot x, g_k \cdot y) = \delta^{\O}_N(x, y) \quad \forall k \in \mathbf{N}.
\end{equation*}
Thus by Remark~\ref{prop_asymp_dist_mark}, we have $g_k \cdot y \rightarrow a$. This holds for all $y \in \O$. For any compact subset $\mathcal{K} \subset \O$, again by Remark~\ref{prop_asymp_dist_mark}, we have $\limhaus g_k\cdot \mathcal{K}=\limhaus \mathcal{K}= \{a\}$ as $k\to+\infty$. Using Fact \ref{fact_dyna}, there exists some $b \in \Ein^{p,q}$ such that $(g_k)$ is $(a,b)$-contracting, in the sense that 
$$g_{k|\Ein^{p,q} \smallsetminus C(b)} \rightarrow a,$$
uniformly on compact subsets of $\Ein^{p,q} \smallsetminus C(b)$. Since $\O$ is a $\Conf(\O)$-invariant proper open subset of $\Ein^{p,q}$, the set 
$$K(\O) := \{x \in \Ein^{p,q} \mid C(x) \cap \O = \emptyset\}$$
is a  $\Conf(\O)$-invariant closed subset of $\Ein^{p,q}$ with nonempty interior. In particular, we can find $x \in K(\O) \smallsetminus C(b)$, so that $g_k \cdot x \rightarrow a \in K(\O)$. This implies $C(a) \cap \O = \emptyset$ by definition of $K(\O)$. $\qedhere$
\end{proof}

\section{Lorentzian features}\label{sect_proof_loren}
This section is devoted to the Einstein universe of Lorentzian signature, meaning that $p=1$. We write $n=q+1$, so that conformal manifolds have signature $(1,n-1)$. In Lorentzian signature, the Einstein universe is endowed with an additional structure$\,$: the \emph{causal structure}. In other words, the Einstein universe $\Ein^{1,n-1}$ is a time-orientable spacetime (see Section \ref{spacetime}). This allows us to use the notion of \emph{causal convexity}. In Sections \ref{convexity} and \ref{Global hyperbolicity}, we introduce this notion and compare it to the notion of dual convexity from Section \ref{sectiondualconvex}. The causal structure provides a direct proof of Theorem \ref{mainth} in Lorentzian signature: see Section \ref{endoftheproof}. For a general proof in arbitrary signature, see Section \ref{sect_proof_any_sign}.

\subsection{Causal convexity}\label{convexity}
In this section we recall the notion of \emph{causal convexity}, specific to conformal Lorentzian manifolds, and compare it to the notion of dual convexity introduced in Section \ref{sectiondualconvex}.

Let $(M,[g])$ be a conformal spacetime (see Section~\ref{spacetime}). Recall that a domain $\O$ of $M$ is said to be \emph{causally convex} if every causal curve of $M$ joining two points of $\O$ lies entirely in $\O$ (cf \cite[pp 8]{Sanchez2008}). A domain $\O\subset M$ is causally convex if and only if for every pair $a,b\in \O$, the \emph{open diamond} $I(a,b)=I^-(a)\cap I^+(b)$ lies entirely in $\O$. This is also equivalent to saying that for every pair $a,b\in \O$, the \emph{closed diamond} $J(a,b)=J^-(a)\cap J^+(b)$ lies entirely in $\O$.
Since $\Ein^{1,n-1}$ is \emph{totally-vicious}, that is, there exists a closed timelike curve through any given point of $\Ein^{1,n-1}$, causally convex domains of $\Ein^{1,n-1}$ are all trivial (see \cite[pp 19]{Sanchez2008}). However, the notion of causal convexity in $\mathbf{R}^{1,n-1}$ is nontrivial (because the Minkowski space is globally hyperbolic, see Section \ref{Global hyperbolicity} for a definition). All our domains will be contained in affine charts, so we make the following definition.

\begin{definition}\label{causallyconv}
    A domain $\O\subset \Ein^{1,n-1}$ is \emph{causally convex} if $\O$ is contained in an affine chart as a causally convex domain of $\mathbf{R}^{1,n-1}$.
\end{definition}
 
One can show that for a domain $\O$ of $\Ein^{1, n-1}$, the property of being causally convex does not depend on the choice of affine chart containing $\O$. Note that this definition of causal convexity \emph{is not} the standard notion of causal convexity in the Einstein universe. 

\begin{prop}\label{dcimpliquecc}
    A dually convex domain (that is not $\Ein^{1,n-1}$) is causally convex.
\end{prop}

\begin{proof}
    Let $\O\subsetneq\Ein^{1,n-1}$ be a dually convex domain. Since $\Omega$ is disjoint from at least one lightcone, we can identify $\O$ with a subset of $\mathbf{R}^{1,n-1}$. Let $x,y \in\O$ and let $\gamma$ be a causal curve of $\R^{1,n-1}$ from $x$ to $y$. Assume by contradiction that $\gamma$ is not contained in $\O$, so that $\gamma$ intersects $\partial \O$ at a point $a\in\partial\Omega$. Since $\O$ is dually convex, there is a point $z\in\Ein^{1,n-1}$ such that $a\in C(z)$ and $C(z)\cap \O= \emptyset$.
    Since $\O \cap C(z)$ is empty, the domain $\O$ lies in one of the connected components of $\mathbf{R}^{1,n-1}\smallsetminus C(z)$. Write $U$ for this component (the number of connected component of $\R^{1,n-1}\smallsetminus C(z)$ is either $2$ or $3$ depending on whether $z\in\partial \R^{1,n-1}$ or $z\in \R^{1,n-1}$, see Figure \ref{Lightcones_in_Minkowski}). Because $U$ is itself causally convex, the curve $\gamma$ is contained in $U$, contradicting $a\in C(z)$.$\qedhere$ 
\end{proof}

\begin{figure}

    $$\begin{array}{cccccc}
    \begin{tikzpicture}[scale = 1.5]
          \draw (-1,-1) -- (1,1); 
          \draw (-1,1) -- (1,-1);
          \draw (-1,-1) arc (180:360:1 and 0.15);
          \draw[dashed] (1,-1) arc (0:180:1 and 0.15);
          \draw (-1,1) arc (180:720:1 and 0.15);
          \fill (0,0) circle (1pt);
          \draw node at (0.3,0){$z$};
          
          \begin{scope}[scale=0.5]
          \draw[dashed] (1,1) arc (0:180:1 and 0.15);
          \end{scope}
          
          \begin{scope}[scale=0.5]
          \draw (-1,1) arc (180:360:1 and 0.15);
          \end{scope}

        \end{tikzpicture}
        & & & & &
          \begin{tikzpicture}[scale = 0.8]
          \begin{scope}[rotate=45]
             \draw (-1,3) -- (1,2) -- (1,-3) -- (-1,-2) -- (-1,3);
             \fill (0,2.8) circle (0.7pt);
             \fill (0,3) circle (0.7pt);
             \fill (0,3.2) circle (0.7pt);
             
             \fill (0,-2.8) circle (0.7pt);
             \fill (0,-3) circle (0.7pt);
             \fill (0,-3.2) circle (0.7pt);
             
             \fill (1.2,-0.6) circle (0.7pt);
             \fill (1.4,-0.7) circle (0.7pt);
             \fill (1.6,-0.8) circle (0.7pt);
             
             \fill (-1.2,0.6) circle (0.7pt);
             \fill (-1.4,0.7) circle (0.7pt);
             \fill (-1.6,0.8) circle (0.7pt);

             \draw [->] (0,0) to (0,1);
             \fill (0,0) circle (1.2pt);
             \draw node at (0.3,0.5){$u$};
             \draw node at (-0.2,-0.2){$x_0$};
          \end{scope}

      \end{tikzpicture}
    
    \end{array}$$

    \caption{We draw the trace in $\R^{1,2}$ of the lightcone of a point $z\in\Ein^{1,2}=\R^{1,2}\cup C(\infty)$. On the left, the point $z$ belongs to $\R^{1,2}$ and the cone $C(z)$ is the set of lightlike vector from $z$. On the right, the point $z$ belongs to $C(\infty)\setminus\{\infty\}$ and $C(z)$ appears as a degenerate affine plane $P=x_0+u^\perp$ in $\R^{1,2}$, for some lightlike vector $u\in\R^{1,2}$.}
    \label{Lightcones_in_Minkowski}
\end{figure}
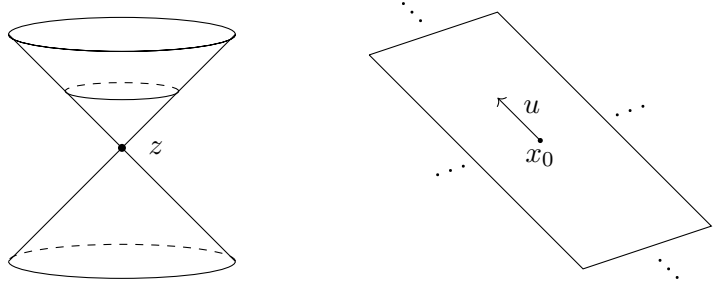

\subsection{Globally hyperbolic spacetimes.} \label{Global hyperbolicity}In this section, we give a complete characterization of proper dually convex domains of the Lorentzian Einstein universe. This section is not used in the rest of the paper. 

Recall that a conformal spacetime $(M,[g])$ is called \emph{globally hyperbolic} if it admits a Cauchy hypersurface $\Sigma$, that is $\Sigma$ is an acausal hypersurface such that any inextensible causal curve contained in $M$ meets $\Sigma$ exactly once (see \cite[Def.\ 3.74]{Sanchez2008}). Given two globally hyperbolic spacetimes $M$ and $N$, a \emph{Cauchy embedding} of $M$ into $N$ is a one-to-one conformal map $f:M\to N$ that sends a Cauchy hypersurface of $M$ to a Cauchy hypersurface of $N$. We say that a globally hyperbolic spacetime $M$ is \emph{maximal} is any Cauchy embedding of $M$ is onto (see \cite[Sect.\ 3]{Clara}). 

Concrete examples of globally hyperbolic spacetimes are causally convex  domains of $\mathbf{R}^{1,n-1}$ (and more generally of $\widetilde{\Ein}^{1,n-1}$). These domains have been studied in \cite{Barbot_2005} and \cite{smaï2023enveloping}. For simplicity we will only describe \emph{bounded} causally convex domains of $\R^{1,n-1}$, or equivalently proper causally convex domains of $\Ein^{1,n-1}$. Let $\O\subset\mathbf{R}^{1,n-1}$ be a bounded causally convex domain. We fix a spacelike hyperplane $H$ in $\mathbf{R}^{1,n-1}$ and we define $V$ to be the image of the orthogonal projection of $\O$ on $H$. If $x\in V$, the normal line to $H$ through $x$ intersects $\O$ in a bounded segment $(f_-(x),f_+(x))$. This defines two 1-Lipschitz functions $f_-,f_+:V\to H^\perp$ which coincide on $\partial V$ such that $$\O=\{(x,t)\in V\times H^\perp\,\vert\,f_-(x)<t<f_+(x)\}.$$
Conversely, any domain defined in this way, where $f_-,f_+:V\to H^\perp$ are 1-Lipschitz functions defined on a bounded domain $V\subset H$ that coincide on $\partial V$, is causally convex. All these domains are seen to be globally hyperbolic. A Cauchy hypersuface of $\O$ is exactly the graph of a 1-Lipschitz function $h:V\to H^\perp$ such that $f_-<h<f_+$. Moreover we know precisely  when these domains are maximal as globally hyperbolic spacetimes: with the notation above, the domain $\Omega$ is maximal exactly when $f_+$ and $f_-$ are \emph{eikonal} (see \cite[Prop.\ 15]{smaï2023enveloping}). We make the link between the ``abstract'' notion of dual convexity and the more concrete and understood notion of causally convex maximal domains of the Minkowski space.
\begin{prop}
    Let $\O$ be a proper causally convex domain of $\Ein^{1,n-1}$. Then $\O$ is dually convex if and only if $\Omega$ is globally hyperbolic maximal. 
\end{prop}

\begin{proof}
    We fix a suitable stereographic projection and identify $\O$ with a bounded subset of $\R^{1,n-1}$. We write $\R^{1,n-1}=H\oplus H^\perp$ for some spacelike hyperplane $H\subset \R^{1,n-1}$ and $\O=\{(x,t)\in V\times H^\perp\,\vert\,f_-(x)<t<f_+(x)\}$ where $V\subset H$ is a bounded domain and $f_-,f_+:V\to H^\perp$ are two 1-Lipschitz functions that coincide on $\partial V$. We denote by $f:\partial V\to H^\perp$ the common value of $f_+$ and $f_-$ on $\partial V$ and $\pi:\R^{1,n-1}\to H$ the projection map.
    
    If $\O$ is globally hyperbolic maximal, then (see \cite[Sect.\ 7.1]{smaï2023enveloping}) $\O$ is a connected component of the complementary $\R^{1,n-1}\smallsetminus C(\Lambda)$, where $\Lambda$ is the graph of $f$ and $C(\Lambda)=\cup_{x\in\Lambda}C(x)$. Therefore $\O$ is dually convex. 
    
    Assume now that $\O$ is dually convex. Let $\O^\prime=\R^{1,n-1}\smallsetminus C(\Lambda)$ where $\Lambda$ is the graph of $f$. Since $\O$ is causally convex, one has $\O\subset\O^\prime$. We want to show that $\partial\Omega\subset C(\Lambda)$. Let $y\in \partial \O$ and let $z\in\Ein^{1,n-1}$ such that $C(z)$ is a supporting lightcone at $y$. If $y\in\Lambda$ then in particular $y\in C(\Lambda)$. Assume $y\not\in\Lambda$, so that $y=(x,f_+(x))$ for some $x\in V$ (the case $y=(x,f_-(x))$ is similar). If $y\neq z$, let $\Delta$ be the unique photon through $y$ and $z$ (if $y=z$, take $\Delta$ to be any photon through $y$). The intersection $\Delta\cap\overline{\Omega}$ is a lightlike segment containing $y$, we write it $[a,b]$ with $a\in I^+(b)$. The projection $\pi([y,b])=[x,\pi(b)]$ is a segment contained in $\overline{V}$. Assume by contradiction that this segment is contained in $V$. Let $v\in H$ be a unit vector collinear to $\pi(b)-x$ and $v^\prime\in H^\perp$ be a unit future directed vector. For $\varepsilon $ sufficiently small one has $f_-(b+\varepsilon v)<f_+(\pi(b))-\varepsilon v^\prime\leq f_+(b+\varepsilon v)$ so $[x,b+\varepsilon(v+v^\prime)/\sqrt{2}]\subset\overline{\O}$, a contradiction with the definition of $[a,b]$.
    
    Therefore $\pi([y,b])$ intersects $\partial V$. We can always shorten $[y,b]$ and assume that $\pi([y,b))$ is contained in $V$ and $\pi(b)\in\partial V$. Therefore $y\in C(b)\subset C(\Lambda)$. This implies that $\O$ is a connected component of $\R^{1,n-1}\smallsetminus C(\Lambda)$, so $\O$ is globally hyperbolic maximal (see \cite[Sect.\ 7.1]{smaï2023enveloping}).$\qedhere$
\end{proof}

\subsection{A short proof of Theorem \ref{mainth} in Lorentzian signature} \label{endoftheproof} We consider the case $p = 1$ and $n :=q+1$. The key point here is the causal structure of $\R^{1, n-1}$. Also, we combine Propositions~\ref{propreetqhimpliquedc} and~\ref{dcimpliquecc} to conclude that proper aslmost-homogeneous domains of $\Ein^{1, n-1}$ are causally convex.

\begin{lem}
    \label{existence_point_extremal}
    Let $\O\subset \mathbf{R}^{1,n-1}$ be a bounded domain. Then for all $x \in \O$, the past $I^-(x)$ 
 and the future $I^+(x)$ both contain at least one photon-extremal point of $\O$. 
\end{lem}

\begin{proof}
    Let $x\in \O$ and write $\mathbf{H}_\lambda$ for the past hyperbolic sheet: $$\mathbf{H}_\lambda=\left\{z\in\R^{1,n-1}\,\vert\,\b(z-x,z-x)=-\lambda^2\right\}\cap I^-(x).$$ Since $\O$ is proper, the sheet $\mathbf{H}_\lambda$ is disjoint from $\O$ for sufficiently large $\lambda$. Let $\lambda_0>0$ be the smallest $\lambda$ with this property. Then $\mathbf{H}_{\lambda_0}$ must intersect $\overline{\O}$. Let $a_0 \in \partial \O\cap\mathbf{H}_{\lambda_0}$. The causal past $J^-(a_0)$ is contained in the past of $\mathbf{H}_{\lambda_0}$, so it intersects $\overline{\O}$ only at $\{a_0\}$. In particular $a_0$ is an extremal point in $\O$. An analogous proof shows that there exists an extremal point $b_0 \in I^+(x) \cap \partial \O$. $\qedhere$
\end{proof}

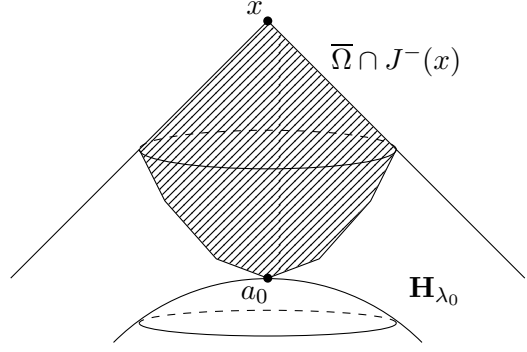
\begin{figure}
    \centering
    \begin{tikzpicture}[scale = 1.7]

          \draw [draw=white,fill, pattern = north east lines, opacity=0.5] (-1,-1) -- (-0.8,-1.4) -- (-0.4,-1.85) -- (0,-2) -- (0.4,-1.85) -- (0.8,-1.4) -- (1,-1) -- (0,0) --(-1,-1);
          \draw (-1,-1) -- (-0.8,-1.4) -- (-0.4,-1.85) -- (0,-2) -- (0.4,-1.85) -- (0.8,-1.4) -- (1,-1) ;
          \draw[dashed] (1,-1) arc (0:180:1 and 0.15);
          \draw (-1,-1) arc (180:360:1 and 0.15);
          \draw node at (1,-0.3) {$\overline{\O}\cap J^-(x)$};  
          \draw node at (-0.1,0.1) {$x$};
          \fill[fill=black] (0,0) circle (1pt);

          \draw[dashed] (1,-2.35) arc (0:180:1 and 0.1);
          \draw (-1,-2.35) arc (180:360:1 and 0.1);

            \draw    (-1.2,-2.5) to [out=45,in=180-45] (1.2,-2.5);

        \draw node at (-0.1,-2.13) {$a_0$};
          \fill (0,-2) circle (1pt);
          \draw node at (1.3,-2.1) {$\mathbf{H}_{\lambda_0}$};
        \draw (-2,-2) -- (0,0); 
          \draw (2,-2) -- (0,0);
        \end{tikzpicture}

    \caption{Existence of extremal points in the past of $x$.}
    \label{fig_extremal_points}
\end{figure}

We can now end the proof of Theorem~\ref{mainth} in the Lorentzian case. Let $\O$ be a proper almost-homogeneous domain of $\Ein^{1,n-1}$, and let us fix once and for all some affine chart $\mathbf{R}^{1,n-1}$ containing $\overline{\O}$. Let $x \in \O$ and let $a_0$ (resp.\ $b_0$) be an extremal point of $\O$ in the past of $x$ (resp.\ in the future), both provided by Lemma~\ref{existence_point_extremal}. Then one has $b_0 \in I^+(a_0)$ and one can consider the diamond $\Diams( a_0, b_0) \subset \mathbf{R}^{1,n-1}$, which contains $x$.  By Proposition \ref{extremal_implique_conedisjoint}, one has $C(a_0) \cap \O = C(b_0) \cap \O = \emptyset$. Since $x \in \Diams(a_0, b_0)$, the domain $\O$ is contained in $\Diams(a_0, b_0)$ by connectedness. Conversely, since $\O$ is causally convex by Propositions~\ref{propreetqhimpliquedc} and~\ref{dcimpliquecc}, one has $\Diams(a_0,b_0) \subset \O$. Hence $\O=\D(a_0, b_0)$. This ends the proof of Theorem~\ref{mainth} in the Lorentzian case.

\section{Proof of the main theorem in higher signature}\label{sect_proof_any_sign}

In this section, we give a proof of Theorem~\ref{mainth} in any signature. For this we use the projective model of the Einstein universe. 

\subsection{A projective lemma}\label{sect_proj_lem} In this section we prove Lemma~\ref{lem_proj_lem}. It is a necessary step to make the link between the pseudo-Riemannian geometry of $\Ein^{p,q}$ and convex geometry inside $\P(\R^{p+1,q+1})$.

First, let us fix some notations. Let $\O\subset\Ein^{p,q}$ be a proper domain: there exists $\xi_0 \in \Ein^{p,q}$ such that $C(\xi_0) \cap \overline{\O} = \emptyset$. Then $\overline{\O} \cap \xi_0^{\perp} = \emptyset$, which means that $\O$ is proper in~$\mathbf{P}(\mathbf{R}^{p+1,q+1})$. Since it is connected, one can lift it to a proper cone $\widetilde{\O} \subset \mathbf{R}^{p+1,q+1}$. Then we define $\widetilde{\mathcal{O}}_{\O} := \mathrm{Conv}(\widetilde{\O})$ its convex hull in $\mathbf{R}^{p+q+2}$, which is a proper convex cone of $\mathbf{R}^{p+q+2}$, a priori not necessarily open. We define $\mathcal{O}_{\O} := \mathbf{P}(\widetilde{\mathcal{O}}_{\O})$. It is a properly convex subset of $\mathbf{P}(\mathbf{R}^{p+1,q+1})$, and it does not depend on the choice of affine chart containing $\overline{\O}$. In particular, $\mathcal{O}_{\O}$ is $\Conf(\O)$-invariant. More precisely, for all $g \in \Conf(\O)$, there exists a unique lift $\widetilde{g} \in \operatorname{O}(p+1,q+1)$ that preserves $\widetilde{\O}$ and hence~$\widetilde{\mathcal{O}}_{\O}$. Each time we will write $\widetilde{g}$ for $g \in \Conf(\O)$ or $\widetilde{\mathcal{O}}_{\O}$ it will be by means of a chosen lift $\widetilde{\O}$ of $\O$.

\begin{lem}\label{lem_proj_lem}
    Let $\O\subset\Ein^{p,q}$ be a proper domain. If $p\geq 1$, then $\mathcal{O}_{\O}$ is a proper domain of $\mathbf{P}(\R^{p+1,q+1})$. 
\end{lem}

\begin{rmk}
    If $p=0$, then the convex hull of $\O$ need not be open.
\end{rmk}

\begin{proof} We already know that $\mathcal{O}_{\O}$ is proper. To prove that $\mathcal{O}_{\O}$ is open, we show that it is a neighborhood of $\O$ in $\mathbf{P}(\R^{p+1,q+1})$: this implies that there is an open set $\O\subset U\subset \mathcal{O}_{\O}$, so in particular $\mathcal{O}_{\O}=\mathrm{Conv}(U)$ is open in $\mathbf{P}(\R^{p+1,q+1})$. Given $z\in \O$, we want to find a neighborhood of $z$ contained in $\mathcal{O}_{\O}$. We can always make $\O$ smaller and assume that $\O=\Diams^{p,q}$. Since diamonds are homogeneous, we can always make a specific choice for $z$. We decompose $\R^{p+1,q+1}=\R^{p,1}\oplus\R^{1,q}$ such that 
    $$\O=\left\{\mathbf{P}(v+w)\,\left\vert\,
         \begin{tabular}{lll}
         $v\in \R^{p,1}$,& $\b(v,v)=+1$,& $v_{p+1}>0$  \\
         $w\in \R^{1,q}$,& $\b(w,w)=-1$,& $w_{q+1}>0$ 
    \end{tabular}\right.\right\}.$$
    Here $v_{p+1}$ (resp.\ $w_{q+1}$) stands for the coordinates of $v$ (resp.\ $w$) in the decomposition $\R^{p,1}=\langle e_1,\dots,e_{p+1} \rangle$ (resp.\ $\R^{1,q}=\langle f_1,\dots,f_{q+1}\rangle$) for which the bilinear form $\b$ is written $\b(v,v)=-(v_1)^2-\dots-(v_p)^2+(v_{p+1})^2$ (resp.\ $\b(w,w)=(w_1)^2+\dots+(w_p)^2-(w_{q+1})^2$). We let $u=e_{p+1}+f_{q+1}\in\O$ and 
    $$\left\lbrace\begin{tabular}{rcll}
        $a_i^{\pm}$ & $=$ & $f_{q+1}+\sqrt{2}e_{p+1}\pm e_i\in\widetilde{\O}$, & $(i=1,\dots,p)$, \\
        $b_j^{\pm}$ & $=$ & $e_{p+1}+\sqrt{2}f_{q+1}\pm f_j\in\widetilde{\O}$, & $(j=1,\dots,q)$.
    \end{tabular}\right.$$
    One can check that $\mathcal{S}=\{a_i^{\pm},b_j^{\pm}\}_{i,j}$ generates $\R^{p+1,q+1}$ and that $u=\frac{1}{(1+\sqrt{2})2p}\sum_ia_i^{\pm}+\frac{1}{(1+\sqrt{2})2q}\sum_jb_j^{\pm}$. Since $u$ is a nontrivial barycenter of $\mathcal{S}$, it belongs to the interior of $\mathrm{Conv}(\mathcal{S})$. Hence the point $z := \mathbf{P}(u)$ is in the interior of $\mathbf{P}(\mathcal{S}) \subset \mathcal{O}_{\O}$. $\qedhere$
\end{proof}

\begin{center}
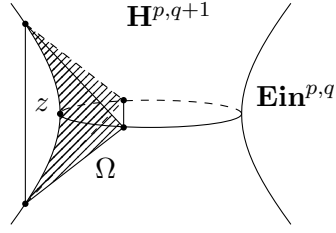
\begin{figure}
    \centering
        \begin{tikzpicture}[scale=1.2]
    \pgfmathsetmacro{\e}{1.44022}
    \pgfmathsetmacro{\a}{1}
    \pgfmathsetmacro{\b}{(\a*sqrt((\e)^2-1)} 

    \fill[fill, pattern = north east lines, opacity=0.5] ({-\a*cosh(+0.85)},{\b*sinh(+0.85)}) -- ({-\a*cosh(-0.85)},{\b*sinh(-0.85)}) -- (-0.3,-0.15);
    \fill[fill, pattern = north east lines, opacity=0.5] ({-\a*cosh(+0.85)},{\b*sinh(+0.85)}) -- ({-\a*cosh(-0.85)},{\b*sinh(-0.85)}) -- (-0.3,+0.15);
    \fill[white] plot[domain=-1:1] ({-\a*cosh(\x)},{\b*sinh(\x)});

    \draw plot[domain=-1:1] ({\a*cosh(\x)},{\b*sinh(\x)});
    \draw plot[domain=-1:1] ({-\a*cosh(\x)},{\b*sinh(\x)});
    \draw[dashed] (1,0) arc (0:180:1 and 0.15);
    \draw (-1,0) arc (180:360:1 and 0.15);
    \draw node at (0.2,1.1) {$\H^{p,q+1}$};  
    \draw node at (1.6,0.2) {$\Ein^{p,q}$};

    \draw (-0.3,0.15) -- (-0.3,-0.15); 
    \draw ({-\a*cosh(-0.85)},{\b*sinh(-0.85)}) -- ({-\a*cosh(0.85)},{\b*sinh(0.85)}); 
    \draw ({-\a*cosh(-0.85)},{\b*sinh(-0.85)}) -- (-0.3,-0.15);
    \draw[dashed] ({-\a*cosh(-0.85)},{\b*sinh(-0.85)}) -- (-0.3,+0.15);
    \draw ({-\a*cosh(+0.85)},{\b*sinh(+0.85)}) -- (-0.3,-0.15);
    \draw[dashed] ({-\a*cosh(+0.85)},{\b*sinh(+0.85)}) -- (-0.3,+0.15);
    
    \fill[fill=black] ({-\a*cosh(-0.85)},{\b*sinh(-0.85)}) circle (1pt);
    \fill[fill=black] ({-\a*cosh(+0.85)},{\b*sinh(+0.85)}) circle (1pt);
    \fill[fill=black] (-0.3,-0.15) circle (1pt);
    \fill[fill=black] (-0.3,0.15) circle (1pt);

    \fill[fill=black] (-1,0) circle (1pt);
    \draw node at (-1.2,0.1) {$z$}; 
    \draw node at (-0.5,-0.6) {$\O$};

\end{tikzpicture}
    \caption{The convex hull of $\O$ contains $z$ in its interior.}
    \label{convex_hull}
\end{figure}
\end{center}

\subsection{Spacelike and timelike-extremal points}
\label{spacelike_timelike_etremal} Photon-extremal points of a proper domain of $\Ein^{p,q}$ have been defined in Section \ref{sect_extr_are_visual}. We will see in the present section that almost-homogeneity makes it possible to separate the set of photon-extremal points into two $\Conf(\O)$-invariant families: namely, the one \emph{spacelike-extremal} points and \emph{timelike-extremal} points. 

We first check that $\Extr(\O)$ is generic, in the sense that lines in $\R^{p+q+2}$ corresponding to photon-extremal points of a proper domain $\O$ of $\Ein^{p,q}$ generate $\mathbf{R}^{p+q+2}$ as a vector space:

\begin{lem}\label{exist_points_extremaux_pq} 
    The extremal points (in the projective sense) of $\mathcal{O}_{\O}$ are all photon-extremal points of $\O$. In particular, there exist $x_1,...,x_{p+q+2} \in \Extr (\O)$ such that $\mathbf{P}(\mathbf{R}^{p+1,q+1}) = x_1 \oplus... \oplus x_{p+q+2}$.
\end{lem}

\begin{proof}
    Let $a \in \partial \mathcal{O}_{\O}$ be an extremal point. By definition of the convex hull, the point $a$ must lie in $\overline{\O}$. If it is in $\O$, then for any photon $\Delta$ passing through $a$, the point $a$ lies in the interior of a connected component of $\overline{\O} \cap \Delta$, which is an open interval of projective line seen in $\mathbf{P}(\mathbf{R}^{p+1, q+1})$, contained in $\overline{\mathcal{O}_{\O}}$. This is impossible since $a$ is an extremal point of $\mathcal{O}_{\O}$. This proves that $a \in \partial \O$. Similarly, if $\Delta$ is a photon, the point $a$ cannot lie on a nontrivial connected component of $\operatorname{int}_\Delta(\overline{\O} \cap \Delta)$, otherwise it would lie in the nontrivial interval of $\operatorname{int}_\Delta (\overline{\mathcal{O}_{\O}} \cap \Delta)$. Hence $a\in\partial\O $ is photon-extremal. $\qedhere$
\end{proof}

From now on, we assume that $\O$ is a proper almost-homogeneous domain of $\Ein^{p,q}$. We fix an affine chart $\mathbf{R}^{p,q}$ containing $\overline{\O}$. For any point $x\in \R^{p,q}$, we write $$I(x)=\{y\in\R^{p,q}\,\vert\,y-x\text{ is timelike vector}\},$$ and $J(x)=\overline{I(x)}\subset\R^{p,q}$. Let $a\in\Extr(\O)$ be a photon-extremal point of $\O$. By Proposition~\ref{extremal_implique_conedisjoint}, one has $C(a) \cap \O = \emptyset$. Since $\O$ is connected, one has $\O \subset I(a)$, in which case we say that $a$ is a \emph{timelike-extremal point}, or $\O \subset \R^{p,q}\smallsetminus J(a)$, in which case we say that $a$ is a \emph{spacelike-extremal point}. We denote by $E^+(\O)$ (resp.\ $E^-(\O)$) the set of spacelike (resp.\ timelike) extremal points of $\O$.

\begin{lem}\label{prop_des_pts_extremaux_espace_et_temps}
The sets $E^+(\O)$ and $E^-(\O)$ satisfy the following properties:
\begin{enumerate}
    \item They are $\Conf(\O)$-invariant.
    \item They are both nonempty.
    \item If $a \in E^+(\O)$ and $b \in E^-(\O)$, then $a \in C(b)$.
\end{enumerate}
\end{lem}

\begin{proof}
    Since the definition of timelike (resp.\ spacelike) extremal points depends of the affine chart, it is not clear that it is $\Conf(\O)$-invariant, as the affine chart $\mathbf{R}^{p,q}$ need not be $\Conf(\O)$-invariant. However, one can write 
    $$E^+(\O)=\left\{a\in\Extr(\O)\,\vert\, \forall U\subset\Ein^{p,q}\text{ neighbourhood of }a,\,I(a)\cap U\cap\O\neq\emptyset\right\},$$
    which is a chart-free expression invariant under the conformal group of $\O$. Similarly the set $E^-(\O)$ is invariant. This proves $1.$

    Let us prove 2. Assume for example that $E^+(\O) = \emptyset$. Then for any two points $a, b \in \partial\widetilde{\mathcal{O}}_{\O} \smallsetminus  \{ 0 \}$ such that $\mathbf{P}(a), \mathbf{P}(b)$ are extremal points of in $\mathcal{O}_{\O}$, one has $\b(a,b) \leq 0$, by definition of $E^-(\O)$. By bilinearity of $\b$ and by definition of $\widetilde{\mathcal{O}}_{\O}$, one has $\b(v,w) \leq 0$ for all $v,w \in \overline{\widetilde{\mathcal{O}}_{\O}}$. By Lemma~\ref{lem_proj_lem}, this is in particular true for all $v,w \in \widetilde{\O}$, which means that for all $x \in \O$ one has $\O \subset \overline{I(x)}$. This is impossible by openness of $\O$ in~$\Ein^{p,q}$. Hence $E^+(\O) \ne \emptyset$, and similarly $E^-(\O) \ne \emptyset$.

    To prove 3, assume $a \in E^+(\O)$ and $b \in E^-(\O)$. Since $a\in\partial\O$ and $\O \subset I(b)$, we get $a\in J(b)$. Similarly $a\in\partial\O$ and $\O \subset \R^{p,q}\smallsetminus J(a)$, so that $b\in \R^{p,q}\smallsetminus I(a)$ or equivalently $a\in \R^{p,q}\smallsetminus I(b)$. Therefore $a \in C(b)$.$\qedhere$
\end{proof}
\subsection{End of the proof of Theorem~\ref{mainth}} \label{sect_end_mainth} We can now complete the proof of Theorem \ref{mainth} in general signature. We will make use of the following lemma, which is already proved in \cite[Lem.\ 3.3]{DGKproj} (and stated with a discreteness assumption which is not necessary). We give its proof for the reader's convenience.

\begin{lem}\label{limit_set_inclus_dans_H} Let $V$ be a finite-dimensional real vector space. We fix $\Vert .\Vert $ any norm on $V$. Let $U \subset V$ be a properly convex open cone, and let $H \leq \operatorname{SL}(V)$ be a subgroup preserving $U$. Let $v \in U$ and $(h_k) \in H^{\mathbf{N}}$ such that there exists $a \in \partial \mathbf{P}(U)$ satisfying $\mathbf{P}(h_k \cdot v) \rightarrow a$. Then  $\Vert h_k\cdot v \Vert  \rightarrow + \infty$.
\end{lem}

\begin{proof}
We still denote by $\Vert .\Vert $ the operator norm associated to $\Vert .\Vert $. Let us first show that $\Vert h_k\Vert  \rightarrow + \infty$. Assume by contradiction that $(h_k)$ admits a subsequence with bounded norm. We still denote this subsequence by $(h_k)$. Then $(h_k)$ converges in $\operatorname{End}(V)$ to some $h$. The limit $h$ is in $\operatorname{SL}(V)$ and preserves $U$. Denoting by $\mathbf{P}(g)$ the projection of an element $g \in \operatorname{End}(V)$ in $\mathbf{P}(\operatorname{End}(V))$, one has $\mathbf{P}(h_k) \rightarrow \mathbf{P}(h)$ in $\PGL(V)$. Since $\mathbf{P}(h_k \cdot v) \rightarrow \mathbf{P}(h \cdot v)$, one has $a = \mathbf{P}(h \cdot v)\in \mathbf{P}(U)$, contradiction. Hence $\Vert h_k\Vert  \rightarrow + \infty$. \\
\indent  Let $\varphi$ be a linear form on $V$ such that $\overline{U} \subset \{ \varphi > 0\}$. We may assume that $\varphi(v) = 1$. The set $U \cap \{ \varphi = 1\}$ is bounded; let $\mathcal{K}$ be its boundary. Since $\mathcal{K}$ is compact, there exists some $0 < \varepsilon < 1$ such that for all $w \in \mathcal{K}$, the line through $v$ and $w$ intersects $\mathcal{K}$ in a $w' \ne w$ such that $v = t w+ (1-t)w'$ for some $t \geq \varepsilon$. Then for all $k \in \mathbf{N}$ one has 
\begin{equation*}
    \varphi(h_k \cdot v) \geq \varepsilon \max_{\mathcal{K}} (\varphi \circ h_k).
\end{equation*}
Thus it is sufficient to see that the maximum of $\varphi \circ h_k$ over $\mathcal{K}$ tends to infinity with $k$. Since $U \cap \{ \varphi =1\}$ is the convex hull of $\mathcal{K}$, it suffices to show that the maximum of $\varphi \circ h_k$ over $U \cap \{ \varphi =1\}$ tends to infinity with $k$. Now since $U$ is a cone, it suffices to show that the supremum of $\varphi \circ h_k$ over $U \cap \{ \varphi < 1\}$ tends to infinity with $k$. Since $\overline{U} \subset \{ \varphi >0 \}$, there exists some $\alpha  >0$ such that $\varphi(u) > \alpha ||u||$ for all $u \in \overline{U}$. By openness, there exists some $\beta >0$, such that for all $k \in \N$ there exists $u_k \in U \cap \{ \varphi < 1 \}$ such that $||h_k \cdot u_k|| > \frac{1}{2}|| h_k || $. Then one has:

\begin{equation*}
    \max_{U \cap \{ \varphi < 1\}} (\varphi \circ h_k) \geq \varphi (h_k \cdot u_k) > \alpha ||h_k \cdot u_k||\geq \alpha \beta|| h_k ||  \rightarrow + \infty.\eqno\qedhere
\end{equation*}

\end{proof}

Let $V_\pp := \operatorname{Span}(E^+(\O)) \subset \mathbf{R}^{p+1,q+1}$ and $V_\mm := \operatorname{Span}(E^-(\O)) \subset \mathbf{R}^{p+1,q+1}$. By Lemma~\ref{exist_points_extremaux_pq}, one has $\mathbf{R}^{p+1,q+1} = V_\pp + V_\mm$. Moreover, by Lemma~\ref{prop_des_pts_extremaux_espace_et_temps}.3), the two spaces $V_\pp$ and $V_\mm$ are orthogonal. It follows easily that $V_\pp \cap V_\mm = \{0\}$, so $\mathbf{R}^{p+1,q+1} = V_\pp \oplus V_\mm$. For the same reason, the spaces $V_\pp, V_\mm$ are nondegenerate, meaning that the restriction of $\b$ to $V_{i}$ for $i \in \{ \pp, \mm\}$ has no kernel and is of signature $(p_i, q_i)$, with $p_\pp + p_\mm = p+1$ and $q_\pp + q_\mm = q+1$. Moreover, by Lemma~\ref{prop_des_pts_extremaux_espace_et_temps}, $1)$, the subspace $V_i$ is $\Conf(\O)$-invariant. Therefore we get a $\Conf(\O)$-invariant orthogonal decomposition  $$\mathbf{R}^{p+1,q+1} = V_\pp \oplus V_\mm.$$

Let us recall the notation of Section~\ref{sect_higher_diamonds} and of the proof of Proposition~\ref{ouverts_homogènes_Einpq}; we write 
$$J = \mathbf{P}\left\{v_\pp + v_\mm \in V_\pp \oplus V_\mm \, \vert \, \b(v_{i}, v_i) = 0 \text{ for } i \in \{\pp, \mm\}\right\}$$
and 
$$U_{i} = \mathbf{P}\left\{v_\pp + v_\mm \in V_\pp \oplus V_\mm \, \vert \, -\b( v_\pp, v_\pp) = \b(v_\mm, v_\mm) = (-1)^i \right\},$$ 
for $i\in\{0,1\}$.

Let $\O_\pp := \O \cap U_\pp$, $\O_\mm := \O \cap U_\mm $, and $\O_J := \O \cap J$. These sets are $\Conf(\O)$-invariant. One has either $\O_\pp \ne \emptyset$ or $\O_\mm \ne \emptyset$, because $\O_J$ has empty interior in $\Ein^{p,q}$. Let us assume for example that $\O_\pp \ne \emptyset$. Let $a \in \partial \O_\pp \subset J \cup \partial \O$. If $a \in J$, then $a \in \partial U_0$. If $a \in \partial \O$, then by almost-homogeneity there exist $(g_k) \in \Conf(\O)^{\mathbf{N}}$ and $x \in \O$ such that $g_k \cdot x \rightarrow a$ (see Remark~\ref{prop_asymp_dist_mark}). \\
    \indent Since $x \in \O \subset \mathcal{O}_{\O}$, there exists $(v_\pp, v_\mm) \in V_\pp \times V_\mm$ such that $v_\pp + v_\mm\in \widetilde{\mathcal{O}}_{\O}$ and $x = \mathbf{P}(v_\pp + v_\mm)$. For all $k \in \mathbf{N}$, let $\widetilde{g}_k$ be the unique lift of $g_k$ defined in Subsection~\ref{sect_proj_lem}. Since $\widetilde{g}_k$ preserves $\widetilde{\mathcal{O}}_{\O}$, by Lemma~\ref{limit_set_inclus_dans_H}, one has
    \begin{equation*}
        \Vert  \widetilde{g}_k (v_\pp + v_\mm) \Vert  \rightarrow + \infty,
    \end{equation*}
for some norm $\Vert .\Vert $ on $V$. \\
\indent On the other hand, up to extracting, there exist $w_\pp \in V_\pp$ and  $w_\mm \in V_\mm$ such that $\Vert w_\pp +w_\mm\Vert  = 1$ and $\frac{\widetilde{g}_k ( v_\pp + v_\mm)}{\Vert \widetilde{g}_k ( v_\pp + v_\mm )\Vert } \rightarrow w_\pp + w_\mm$. This implies $a = \mathbf{P}(w_\pp + w_\mm )$. But
\begin{equation*}
    \b(w_\pp, w_\pp) = \lim_{n \rightarrow + \infty} \b\left(\frac{\widetilde{g}_k v_\pp}{\Vert  \widetilde{g}_k( v_\pp + v_\mm)\Vert }, \frac{\widetilde{g}_k  v_\pp }{\Vert  \widetilde{g}_k ( v_\pp + v_\mm)\Vert }\right) = 0, 
\end{equation*}
and the same computation holds for $w_\mm$, meaning that $a \in J \subset \partial U_0$. \\
\indent We have proved that $\partial \O_0 \subset U_\pp$. Hence $\O_\pp$ is closed in $U_\pp$. Since it is also open, it is a union of connected components of $U_\pp$. But as soon as $q_\pp \geq 2$ or $p_\mm \geq 2$, by Proposition~\ref{ouverts_homogènes_Einpq}, the open set $U_\pp$ has no proper connected components. Since $\O_\pp \subset \O$ is proper, this implies that $(p_\pp, q_\pp) =(p,1)$ and $(p_\mm,q_\mm) = (1,q)$. Hence again by Proposition~\ref{ouverts_homogènes_Einpq}, the set $\O_\pp$ is a diamond. \\
\indent If $\O_\mm \ne \emptyset$, then by Proposition~\ref{ouverts_homogènes_Einpq}, it has to be the diamond dual to $\O_\pp$. But then $\O_\pp \cup \O_\mm \subset \O$ is not proper. Hence necessarily $\O_\mm = \emptyset$. By openness of $\O$, one thus have $\O_J = \emptyset$, hence $\O = \O_\pp$ is a diamond. This concludes the proof of Theorem~\ref{mainth}.

\section{Conformally flat manifolds with proper development}\label{sect_proof_mainth2}

This section is devoted to the proof of Theorem~\ref{maincor}.

\begin{proof}[Proof of Theorem~\ref{maincor}]
    Let $\dev:\widetilde{M}\to\Ein^{p,q}$ be a developing map for $M$ and let $\O=\dev(\widetilde{M})$. Since $M$ is closed, there is a compact fundamental domain $\mathcal{K}\subset\widetilde{M}$ for the action of $\pi_1(M)$. Since the developing map is equivariant, the compact set $\dev(\mathcal{K})\subset\O$ intersects any $\hol(\pi_1(M))$-orbit, that is $\O$ is almost-homogeneous. Theorem~\ref{mainth} implies that $\O$ is a diamond and is conformally equivalent to $-\H^p\times\H^q$. Let $g_\O$ be the conformally invariant Riemannian metric of $\O$ equal to $g_{\H^p}\oplus g_{\H^q}$ under this identification and let $g=\dev^*g_\O$. The metric $g$ is invariant under $\pi_1(M)$, so it defines a Riemannian metric of $M$. This metric must be complete since $M$ is closed, so $g$ is also complete. The map $\dev$ is a local isometry between complete Riemannian manifolds, so it is a covering map. Since $\O$ is simply connected, the covering map is a diffeomorphism onto its image and $M$ is a compact quotient of $\H^p\times\H^q$. 
    
When $1\leq p<q$ are distinct and $(p,q)\neq (2,3)$, the group $\PO(p,1)\times\PO(1,q)$ is non-isotypic and it is a general fact that it's cocompact lattices are virtually products (see e.g.\ \cite[Thm.\ 5.6.2]{morris2015} for a proof using Margulis Arithmeticity when $p, q \geq 2$). This fact completes the proof of the theorem. $\qedhere$
\end{proof}

 When $p=q$, the argument in the first part of the proof of Theorem~\ref{maincor} is valid, so any closed conformally flat $(p,p)$-manifold with proper development is conformally equivalent to a quotient $(-\H^p)\times\H^p/\Gamma$, where $\Gamma<\PO(p,1)\times\PO(1,p)$ is a cocompact lattice. However, the group $\PO(p,1)\times\PO(1,p)$ is isotypic, hence it admits irreducible cocompact lattices (see \cite[Cor.\ 18.7.4]{morris2015}). 
 The case $(p,q)=(2,3)$ is also special, as $\PO(1,2) \times \PO(1,3)$ admits cocompact lattices which are not virtually products of cocompact lattices of $\PO(1,2)$ and $\PO(1,3)$, since $\mathfrak{so}(3,\C)$ is an irreducible factor of both $\mathfrak{so}(1,2)_{\C} = \mathfrak{so}(3,\C)$ and $\mathfrak{so}(1,3)_\C\simeq \mathfrak{so}(4,\C)\simeq \mathfrak{so}(3,\C)\oplus \mathfrak{so}(3,\C)$.

\section{Exceptional isomorphisms in low dimensions}\label{sect_except_iso}

\subsection{The Grassmannian $\operatorname{Gr}_2(\mathbf{R}^{4})$ }\label{sect_gr2} In \cite{van2019rigidity}, Limbeek--Zimmer prove that any proper divisible domain of the Grassmannian $\operatorname{Gr}_p(\mathbf{R}^{2p})$ of $p$-planes of $\mathbf{R}^{2p}$ which is \emph{convex in some affine chart} is isomorphic to a model of the symmetric space of $\PO(p,p)$.  
The exceptional isomorphism $\SLq^0 \simeq \PO(3,3)^0$ allows us to strengthen this rigidity result in the case of $\Gr$, not making any convexity assumption and only asking for almost-homogeneity.

The group $\SLq$ acts naturally on the flag manifold $\Gr$, and the stabilizer of a point is a parabolic subgroup of $\SLq$. An \emph{affine chart} of $\Gr$ is a subset of the form 
\begin{equation*}
   \mathbf{A}_W = \{V \in \Gr \mid V \cap W = \{ 0 \} \},
\end{equation*}
where $W \in \Gr$ (see e.g.\ \cite{van2019rigidity}). A domain $\O \subset \Gr$ is \emph{proper} if its closure is contained in an affine chart. Its \emph{automorphism group} is the set of all elements $g \in \SLq$ such that $g \cdot \O = \O$. We say that two domains $\O_1, \O_2 \subset \Gr$ are \emph{isomorphic} if there exists $g \in \SLq$ such that $\O_2 = g \cdot \O_1$. 
\begin{ex}\label{ex_Boule}
    The symmetric space of $\PO(2,2)$ embeds as a proper symmetric domain of $\Gr$. To be more precise, let $\b$ be a nondegenerate symmetric bilinear form of signature $(2,2)$ on $\R^4$ and let $\mathbf{B}_{2,2}(\b)$ be the subset of $\Gr$ consisting of all the $2$-planes $V$ such that $\b_{|V \times V}$ is positive definite. The domain $\mathbf{B}_{2,2}(\b)$ is proper, symmetric and divisible (see \cite{van2019rigidity}), and its automorphism group is equal to $\PO(\b) \simeq \PO(2,2)$, and the stabilizer of a point is a maximal compact subgroup of $\PO(\b)$. Hence $\mathbf{B}_{2,2}(\b)$ is a model for the symmetric space $\mathbf{H}^2 \times \mathbf{H}^2$ of $\PO(2,2)$. By the Sylvester's law of inertia, all the domains $\mathbf{B}_{2,2}(\b)$, with $\b$ of signature $(2,2)$, are isomorphic.
\end{ex}

Let $\tau: \SLq \rightarrow \PGL(\PRSP)$ be the group embedding defined by the action of $\SLq$ on the projectivization of alternate product $\PPRSP$. By factorization, we get the classical \emph{Plücker embedding}
        \begin{equation}\label{eq_pluck}
            \iota: \begin{cases}
                \Gr &\longrightarrow \PPRSP \\
                \operatorname{Span}(u,v) & \longmapsto  \mathbf{P}(u \wedge v).
            \end{cases}
        \end{equation}
The bilinear form $\omega$ defined on $\PRSP$  by 
\begin{equation*}
    \omega(x,y) = x \wedge y \quad \forall x,y \in \PRSP,
\end{equation*}
is nondegenerate, symmetric, of signature $(3,3)$. Since $\tau(\SLq)$ preserves this bilinear form, one has $\tau(\SLq) \subset \PO(\omega)\simeq\PO(3,3)$, with an equality between the identity components of these two groups. Then the image of the Plücker embedding~\eqref{eq_pluck} is equal to 
\begin{equation*}
    \Ein(\PRSP, \omega) := \{ \mathbf{P}(x) \in \PPRSP \mid \omega(x,x) = 0\} \simeq \Ein^{2,2}.
\end{equation*}
Thus there is a $\tau$-equivariant diffeomorphism $\Gr \simeq \Ein^{2,2}$.

\subsection{Triality and $\PO(4,4)$}\label{sect_triality} Another exceptional isomorphism arising in low dimension appears for $\Ein^{3,3}$. The set of maximal totally isotropic subspaces of $\R^{4,4}$ has two connected components, denoted by $\mathcal{F}_{1}$ and $\mathcal{F}_2$. They are both flag manifolds, corresponding to two extremal roots of the Dynkin diagram of $\PO(4,4)$. 
 The root system of $\PO(4,4)$ is $D_4$. It is a tripod and has automorphism the symmetric group $\mathfrak{S}_3$. The extremal roots correspond to the flag manifolds $\Ein^{3,3}$, $\mathcal{F}_1$ and $\mathcal{F}_2$, and are permuted by the automorphism group of $D_4$. In particular, there exists an automorphism $\sigma$ of order $3$, sending the root corresponding to $\Ein^{3,3}$ to the one corresponding to $\mathcal{F}_1$. The automorphism $\sigma$ induces an outer automorphism $\varphi$ of order $3$ of $\PO(4,4)$, called \emph{triality}. Then $\varphi$ induces a $\varphi$-equivariant diffeomorphism $\Ein^{3,3} \simeq \mathcal{F}_1$. The notion of transversality, and hence of properness (defined in a general setting in \cite{Zimpropqh}), is preserved by this diffeomorphism, as all the flag manifolds of $\PO(4,4)$ are symmetric (the opposition involution of $D_4$ is trivial). The same construction holds for $\mathcal{F}_2$, considering $\varphi^2$ instead of $\varphi$.

 \subsection{Proof of Corollary \ref{cor_rigidity}}

$(1)$  Let $ \O  \subset \Gr$ be a proper almost-homogeneous domain. By properness, there exists $y \in \Gr$ such that $y \cap x = \{0\}$ for all $x \in \overline{\O}$. But this is equivalent to saying that
$\iota(x) \notin \iota(y)^{\perp_{\omega}}$ for all $x\in\overline{\O}$. Hence $C(\iota(y)) \cap \overline{\iota(\O)} = \emptyset$, so $\iota(\O)$ is a proper domain of $\Ein^{2,2}$. Moreover, by $\tau$-equivariance of $\iota$, the domain $\iota(\O)$ is $\PO(3,3)$-almost-homogeneous. Then by Theorem~\ref{mainth}, it is a diamond. By the equality $\tau(\SLq^0) =  \PO(3,3)^0$, we thus know that all proper almost-homogeneous domains of $\Gr$ are isomorphic. Since the domain $\mathbf{B}_{2,2}(\b_0)$ (for some fixed $(2,2)$-bilinear form $\b_0$) is part of them, they are all isomorphic to it. Point $(1)$ then follows by Example~\ref{ex_Boule}.

$(2)$ This is a straightforward consequence of the $\varphi$-equivariance (resp.\ $\varphi^2$-equivariance) of the diffeomorphism $\Ein^{3,3} \simeq \mathcal{F}_1$ (resp.\ $\Ein^{3,3} \simeq \mathcal{F}_2$) preserving the notions of transversality and properness. $\qedhere$

\bibliographystyle{alpha}
\bibliography{bibliography}

\end{document}